\documentclass[12pt,final]{amsart}
\usepackage{amssymb, amsmath, amsthm}
\usepackage{mathrsfs}  
\usepackage{graphicx}
\usepackage[colorlinks=true, citecolor=blue, linkcolor=red, urlcolor=blue]{hyperref}
\usepackage[text={6in,8in},centering]{geometry}
 \usepackage{xcolor}
 \usepackage{tcolorbox}
\usepackage{mathrsfs}  

\usepackage{bm}
\usepackage{ifdraft}
\ifoptionfinal{
\usepackage[disable]{todonotes}
}{
\usepackage[norefs, nocites]{refcheck}

\usepackage[notref, notcite]{showkeys}
\usepackage[bordercolor=white, color=white]{todonotes}
}

\makeatletter\providecommand\@dotsep{5}\def\listtodoname{List of Todos}\def\listoftodos{\hypersetup{linkcolor=black}\@starttoc{tdo}\listtodoname\hypersetup{linkcolor=blue}}\makeatother
\usepackage{hyperref}
\usepackage{etoolbox}

\newtheorem{theorem}{Theorem}[section]
\newtheorem{lemma}[theorem]{Lemma}

\newtheorem{proposition}[theorem]{Proposition}
\newtheorem{definition}[theorem]{Definition}

\theoremstyle{remark}

\def\C{\mathbb C}
\def\R{\mathbb R}

\def\N{\mathbb N}

\renewcommand{\leq}{\leqslant}
\renewcommand{\geq}{\geqslant}

\def\p{\partial}
\DeclareMathOperator{\supp}{supp}

\newcommand*\xbar[1]{%
   \hbox{%
     \vbox{%
       \hrule height 0.5pt % The actual bar
       \kern0.5ex%         % Distance between bar and symbol
       \hbox{%
         \ensuremath{#1}%
       }%
     }%
   }%
} 

\title[Partial data inverse problems for heat equations]{Partial data inverse problems for reaction-diffusion and heat equations}

\author[A. Feizmohammadi]{Ali Feizmohammadi}
\address{Department of Mathematical and Computational Sciences, University of Toronto Mississauga, 3359 Mississauga Road
	Deerfield Hall, DH 3027, Mississauga, ON, Canada L5L 1C6}
\email{ali.feizmohammadi@utoronto.ca}

\author[Y. Kian]{Yavar Kian}
\address{UFR des Sciences et Techniques Université de Rouen Normandie, Avenue de l'Universit\'{e}, BP.12 76801 Saint-Etienne-du-Rouvray, France}
\email{yavar.kian@univ-rouen.fr}

\author[G. Uhlmann]{Gunther Uhlmann}

\address
{Department of Mathematics\\
	University of Washington\\
	Seattle, WA  98195-4350\\
	USA}
\email{gunther@math.washington.edu}

\keywords{}
\begin{document}

%\begin{titlepage}
%\maketitle
%\end{titlepage}

\maketitle
\begin{abstract}
We study partial data inverse problems for linear and nonlinear parabolic equations with unknown time-dependent coefficients. In particular, we prove uniqueness results for partial data inverse problems for semilinear reaction-diffusion equations where Dirichlet boundary data and Neumann measurements of solutions are restricted to any open subset of the boundary. We also prove injectivity of the Fr\'{e}chet derivative of the partial Dirichlet-to-Neumann map associated to heat equations. Our proof consists of two crucial ingredients; (i) we introduce an asymptotic family of spherical quasimodes that approximately solve heat equations modulo an exponentially decaying remainder term and (ii) the asymptotic study of a weighted Laplace transform of the unknown coefficient along a straight line segment in the domain where the weight may be viewed as a semiclassical symbol that itself depends on the complex-valued frequency. The latter analysis will rely on Phragm\'{e}n-Lindel\"{o}f principle and Gr\"{o}nwall inequality. 
\end{abstract}

\section{Introduction}
Let $M=(0,T)\times \Omega$ for some $T>0$ and with $\Omega$ denoting  a connected bounded open set in $\R^n$, $n\geq 2$, with smooth boundary.  Let us consider a function $a(t,x,\mu)$ with $\mu\in \R$ and $(t,x)\in \overline{M}$ (closure of $M$) that satisfies 
\begin{itemize}
	\item [{\bf(H)}]{$a(t,x,0)=\p_\mu a(t,x,0)=0$ and $a(t,x,\mu)$ depends real analytically on $\mu$ in the sense of taking values in $C^{\frac{\alpha}{2},\alpha}(\overline{M})$ for some $\alpha \in (0,1)$.}
\end{itemize} 
We consider the reaction-diffusion equation
\begin{equation}\label{heat_semi}
	\begin{aligned}
		\begin{cases}
			\p_t u - \Delta u + a(t,x,u)=0 
			&\text{on  $M=(0,T)\times \Omega$},
			\\
			u=f & \text{on $\Sigma=(0,T)\times \p \Omega$},
			\\
			u|_{t=0}=0 & \text{on $\Omega$}.
		\end{cases}
	\end{aligned}
\end{equation}
Following \cite[Proposition 3.1.]{KLL}, there exists $\kappa>0$ such that given any $f\in \mathcal K_0$, with
$$ \mathcal K_0:=\{h\in C^{1+\frac{\alpha}{2},2+\alpha}(\overline{\Sigma})\,:\, h(0,\cdot)=\p_t h(0,\cdot)=0 \quad \text{and} \quad \|h\|_{C^{1+\frac{\alpha}{2},2+\alpha}(\overline{\Sigma})}\leq \kappa,\}$$
equation \eqref{heat_semi} admits a unique solution $u\in C^{1+\frac{\alpha}{2},2+\alpha}(\overline{M}).$ Let $\Gamma\subset \p\Omega$ be a nonempty open set. We define the local Dirichlet-to-Neumann map associated to the above equation through 
$$ \mathcal N^{\Gamma}_a (f) = \p_\nu u\big|_{(0,T)\times \Gamma} \qquad \forall\, f \in \mathcal K_0\quad \text{with} \quad \supp f \subset (0,T)\times \Gamma,$$
where $u$ is the unique solution to \eqref{heat_semi} with Dirichlet boundary data $f$.
We are interested in the following inverse problem.
\begin{itemize}
	\item [{\bf(IP1)}]{Is it possible to determine the semilinear model \eqref{heat_semi} given the local boundary measurements $\mathcal N_a^{\Gamma}$?}
\end{itemize}
Let us mention that the reaction-diffusion model \eqref{heat_semi} arises frequently in chemistry, biology, geology, physics and ecology, including the dynamics of biological populations \cite{FIS}, the Rayleigh-B\'{e}nard convection \cite{NW} or models appearing in combustion theory \cite{Vol14,YBZFK38}. The inverse problem (IP1) may be viewed as the question of whether or not it is possible to determine the governing physical law inside a medium, given the knowledge of diffusion fluxes on a part of its boundary. 

The literature of inverse problems for (IP1) has mostly been concerned with the case where the observation set $\Gamma$ is equal to the full boundary. The first result in this direction is due to Isakov in \cite{Is1}, assuming that $\Gamma=\p \Omega$ and also that one is allowed to measure data at the final time $t=T$ on the entire domain $\Omega$, see \cite{COY,Is3} for further results and \cite{KiUh} for the state of the art result in the full boundary case of reaction-diffusion equations where the final time measurements are removed. To the best of our knowledge, there are no results on (IP1) that consider an arbitrary open set $\Gamma\subset \p \Omega$. The main goal of this article is to  give a resolution of (IP1) as follows.

\begin{theorem}
	\label{thm_semi}
	Let $M=(0,T)\times \Omega$ where $T>0$ and $\Omega\subset \R^n$ is a connected bounded domain with smooth boundary. Let $\Gamma \subset \p \Omega$ be a nonempty open set. Let $a_1,a_2$ be two functions on $\overline{M}\times \R$ that satisfy (H). Then,
	$$ \mathcal N^\Gamma_{a_1}= \mathcal N^{\Gamma}_{a_2} \implies a_1=a_2.$$
\end{theorem}
The proof of the above theorem will rely on several properties of heat equations such as null-controllability of its solutions from the boundary, maximum principle, parabolic regularity and a key ingredient given by a novel uniqueness result for linearized partial data inverse problems subject to heat equations that we prove in the present article. To state our results in this direction, we start fixing some notations. Let $q\in L^{\infty}(M)$ and consider the following heat equation 
\begin{equation}\label{heat}
	\begin{aligned}
		\begin{cases}
		\p_t u - \Delta u + q(t,x)\,u=0 
			&\text{on  $M$},
			\\
			u=f & \text{on $\Sigma=(0,T)\times \p \Omega$},
			\\
			u|_{t=0}=0 & \text{on $\Omega$}.
		\end{cases}
	\end{aligned}
\end{equation}
For existence and uniqueness of solutions to the above PDE, we will use the standard mixed Sobolev spaces
$$ H^{r,s}(M)= H^r(0,T;L^2(\Omega)) \cap L^2(0,T;H^s(\Omega))$$
and 
$$H^{r,s}(\Sigma)= H^r(0,T;L^2(\Sigma)) \cap L^2(0,T;H^s(\Sigma))$$
for any $r,s\geq 0$. It is classical (see e.g. \cite[Section 2]{LM}) that given any Dirichlet boundary data $f$ in the space
$$ \mathcal H(\Sigma) = \{ h\in H^{\frac{3}{4},\frac{3}{2}}(\Sigma)\,:\, h|_{t=0}=0  \},$$
equation \eqref{heat} admits a unique solution 
$$u\in H^{1,2}(M) \qquad \text{with} \qquad \p_\nu u \in H^{\frac{1}{4},\frac{1}{2}}(\Sigma),$$
where $\nu$ is the exterior unit normal vector field on $\p \Omega$. Let $\Gamma \subset \p\Omega$ be a nonempty open set. We define the partial data Dirichlet-to-Neumann map 
$$\Lambda_q^\Gamma :  \mathcal H(\Sigma) \to  H^{\frac{1}{4},\frac{1}{2}}((0,T)\times \Gamma)  $$ defined as 
$$ \Lambda^{\Gamma}_{q}(f)= \p_\nu u|_{(0,T)\times \Gamma} \qquad \forall\, f\in \mathcal H(\Sigma) \quad \text{with} \quad \supp f \subset (0,T]\times \Gamma,$$
where $u$ is the unique solution to \eqref{heat} subject to Dirichlet boundary data $f$. The following natural inverse problem arises:
\begin{itemize}
	\item [{\bf (IP2)}]{Is the map $q \mapsto \Lambda_q^\Gamma$ injective?}
\end{itemize}

Akin to (IP1), when the observation set $\Gamma$ is equal to the full boundary $\p \Omega$, (IP2) is well understood, as well as more general variations of it that allow variable leading order coefficients in \eqref{heat}, see e.g. \cite{Is0,Is4,CK1,Fei}.  However, results are rather sparse when $\Gamma$ is a proper subset of the boundary. In physical experiments, it is very typical that we only have access to very small subsets of the boundary of a medium and thus it is fundamentally important to address (IP2) with possibly very small subsets of the full boundary. In this direction, we mention that if the observation set $\Gamma$ is relatively large, roughly more than half the size of $\p \Omega$, it is possible to determine the zeroth order coefficient $q$. Such results can be found in \cite{CK1,CK2} where the authors  considered partial data results in the spirit of the earlier work \cite{BU} with, roughly speaking, the excitation $f$ restricted to one half of the boundary and measurements of the flux on the other half of the boundary. In the same spirit, we can mention the recent work of \cite{KLL} who considered (IP1) and (IP2) with restriction of the flux to a part of the boundary illuminated by an arbitrary chosen point outside the domain, as considered by the earlier work \cite{KSU} for the Calder\'on problem.  We also mention that if the coefficient $q=q(x)$ is assumed to be independent of time, uniqueness is known for arbitrary open sets $\Gamma$ \cite{CK}. The latter result uses the time-independence of the coefficients to transform the inverse problem for a more general version of \eqref{heat} to the Gel'fand inverse spectral problem \cite{Gelfand} that has been resolved in \cite{Bel1,BK92}, see also the survey article \cite{Bel2}. 

In this paper, as a first step of understanding (IP2) for arbitrary open sets $\Gamma\subset \p \Omega$, we propose to consider the linearized version of (IP2) by considering the Fr\'{e}chet derivative of the Dirichlet-to-Neumann map $\Lambda_q$ evaluated at $q=0$. Precisely, for each $q\in L^{\infty}(M)$ we define
$$ \mathcal S_q^\Gamma = \frac{d}{ds}\Lambda_{sq}^\Gamma\, \quad \text{evaluated at $s=0$},$$
and consider the following inverse problem:
\begin{itemize}
	\item [{\bf (IP3)}]{Is the map $q \mapsto \mathcal S_{q}^\Gamma$ injective?}
\end{itemize}
The study of linearized versions of Dirichlet-to-Neumann maps is important as it may shed some light on the original inverse problem. To further illustrate this view, we recall the well known Calder\'{o}n problem in electrical impedance tomography that is concerned with recovering electrical conductivity of a medium, given voltage and current measurements on its boundary, see e.g. \cite{Uhl} for a survey of this well known inverse problem. Akin to (IP2), in the context of the Calder\'{o}n problem, complete results are available when the observation set $\Gamma$ is equal to the full boundary but much less is known about the partial data problems $\Gamma\subset \p \Omega$ in dimensions three and higher, see e.g. \cite{KS} for the state of the art result. As such, many authors have looked at linearized versions of the Calder\'{o}n problem with the hope to gain a better insight into the fully nonlinear problem. Indeed, Calder\'{o}n's original paper on the subject gave a resolution of the linearized version with full boundary data, by introducing a family of explicit harmonic functions known as complex geometric optics \cite{Calderon}. The resolution of Calder\'{o}n's problem for the case of full boundary data ultimately relied fundamentally on the notion of generalizations of these special solutions \cite{SU}. We also mention the works \cite{DKSU} and \cite{SU} on linearized Calder\'{o}n problem with measurements on arbitrary small subsets of the boundary as well as the works \cite{GH} on numerical analysis of such problems and finally \cite{DKLLS, KLS,GST} for consideration of the full data linearized Calder\'{o}n problem on more general Riemannian or complex manifolds. We also refer the reader to \cite{KuSa} for the study of linearized Calder\'{o}n problem for polyharmonic elliptic operators. 

Motivated by such considerations in the Calder\'{o}n problem, we give a resolution of (IP3) under the assumption that $q$ is known in a neighbourhood of initial and final times, namely on the set 
$$ \mathcal D_\delta= \{(t,x) \in (0,T)\times \Omega\,:\, t<\delta \quad \text{or} \quad t>T-\delta \},$$
for some $\delta>0$. Precisely, we prove the following theorem. 
\begin{theorem}
	\label{thm_lin}
	Let $M=(0,T)\times \Omega$ where $T>0$ and $\Omega\subset \R^n$ is a connected bounded domain with smooth boundary. Let $\Gamma \subset \p\Omega$ be a nonempty open set. Let $q_1,q_2\in L^{\infty}(M)$ satisfy $q_1=q_2$ on $\mathcal D_\delta$ for some $\delta>0$. Then,
	$$ \mathcal S^\Gamma_{q_1}= \mathcal S^{\Gamma}_{q_2} \implies q_1=q_2.$$
\end{theorem}

Let us remark that the assumption on the coefficient $q$ being known in a neighbourhood of initial and final times may be relaxed to it being independent of time in a small neighbourhood of initial and final times. The proof of Theorem~\ref{thm_lin} will rely on several arguments such as approximate controllability for heat equations and a  key ingredient given by  a completeness result for products of solutions to heat equations that we prove also in the present article. This result about completeness  for products of solutions to heat equations can be stated as follows.

\begin{theorem}
	\label{thm_density}
		Let $M=(0,T)\times \Omega$ where $T>0$ and $\Omega\subset \R^n$ is a connected bounded domain with smooth boundary. Let $\Gamma\subset \p\Omega$ be a nonempty open set and let $p\in \Gamma$. Let $q\in L^{\infty}(M)$ satisfy $q=0$ on $\mathcal D_\delta$ for some $\delta>0$. Suppose that 
	\begin{equation}\label{density_eq} \int_\delta^{T-\delta}\int_\Omega q(t,x)\, u_1(t,x)\,u_2(t,x) \,dx\,dt =0,\end{equation}
	for any pair $u_1\in C^{\infty}([\delta,T]\times \overline\Omega)$ and $u_2 \in C^{\infty}([0,T-\delta]\times \overline\Omega)$ that respectively satisfy 
	$$ \p_t u_1 -\Delta u_1=0 \quad \text{on $(\delta,T)\times \Omega$} \quad \text{with} \quad \supp (u_1|_{[\delta,T]\times \p \Omega})\subset [\delta,T]\times \Gamma,$$
	and 
	$$ \p_t u_2 +\Delta u_2=0 \quad \text{on $(0,T-\delta)\times \Omega$} \quad \text{with} \quad \supp (u_2|_{[0,T-\delta]\times \p \Omega})\subset [0,T-\delta]\times \Gamma.$$
	Then $q=0$ on $M$.
\end{theorem}
 
Let us remark that a possible choice for $u_1$ and $u_2$ in the above theorem is harmonic functions whose boundary traces are supported on the set $\Gamma$. Under such a choice and by applying the main result of \cite{DKSU} one can immediately deduce that the integral identity \eqref{density_eq} yields that the time average of the unknown function $q$ must be zero at each spatial point, that is to say, 
$$ \int_{0}^T q(t,x)\,dt=0 \quad \forall\, x\in \Omega.$$
Clearly, this is not sufficient to conclude that $q$ must vanish everywhere. Indeed, our proof of the above completeness property will rely partly on introducing a suitable family of solutions to heat equations that we will term as {\em spherical quasimodes} as well as the careful construction of remainder terms for such quasimodes. Although some of the arguments in the proof of Theorem~\ref{thm_density} are similar in spirit to analytic microlocal analysis used for solving the linearized Calder\'{o}n problem (see e.g.  \cite{DKSU,SU}), there are also key distinctions. For instance, in the case of the heat equation \eqref{heat} one does not have analogues of Kelvin transform for harmonic functions that can transform the domain into one with locally convex boundary. The latter transform was an important tool in the elliptic linearized Calder\'{o}n problem studied in \cite{DKSU} as it allowed the authors to work with explicit family of complex geometric optics. To overcome such a difficulty, we introduce an exponential quasimode construction for heat equations with real valued phase and amplitude terms whose level sets are spheres, allowing us to penetrate from any point of the boundary, irrespective of its lack of convexity. We do not impose analyticity on the boundary of the domain either. Another key difference is that although in the works \cite{DKSU,SU} on the linearized Calder\'{o}n problem, the analysis eventually relates to the Fourier--Bros--Iagolnitzer transform of the unknown coefficient in the phase space, in the case of the heat equation, (since the weight of the exponential quasimode is purely real-valued), our analysis eventually reduces to the study of a weighted Laplace transform along a geodesic of the form
$$ \int_I q(r) e^{\tau t}(a_0(r) + \tau^{-1} 
a_1(r)+\tau^{-2}a_2(r)+...)dr$$
where $I$ is an interval on a line segment with one end on the boundary of the domain and $\{a_k(r)\}_{k=0}^{\infty}$ is a formal analytic symbol satisfying the bounds $\|a_k\|_{L^{\infty}(I)}\leq C^{1+k}\,k!$. We need to analyze the asymptotic properties of the latter weighted Laplace transform in some detail. This may be viewed as the crux of the argument in the proof of a local version of Theorem~\ref{thm_density} that is proved in Section~\ref{sec_local}. We believe that the analysis of the latter transform is new in inverse problems which reduces to a combination of Phragm\'{e}n-Lindel\"{o}f principle and Gr\"{o}nwall inequalities to uniquely recover $q$ pointwise in the interior of the domain.

We close this section by mentioning that the link between (IP1) and (IP2) lies in a linearization method that was first introduced by Isakov in \cite{Is1}. An important variant of this method known as {\em higher order linearization} method was introduced in \cite{KLU} by Kurylev, Lassas and the third author in the context of hyperbolic equations. We mention that in recent years, linearization methods have been an extremely successful tool in the study of inverse problems for nonlinear equations, see for example \cite{FKU,FO,KLL, KiUh,KU1,KU2,LLLS,LLLS2} and the references therein. 
 
\subsection{Organization of the paper}
We begin with Section~\ref{sec_prelim} that is concerned with understanding the linearized Dirichlet-to-Neumann map, in the sense of the direct problem. In Section~\ref{sec_local} we will prove a local version of the completeness property for products of solutions to linear heat equations with supports restricted to arbitrary portions of the boundary. We will begin by introducing a family of spherical quasimode solutions to heat equations that will be utilized in the proof of the completeness property followed by a careful asymptotic analysis of products of such quasimodes. In Section~\ref{sec_global}, we will use Runge's approximation property to march into the domain and turn the local completeness result to a global one, thus completing the proof of Theorem~\ref{thm_density} as well as Theorem~\ref{thm_lin}. Finally, we will prove Theorem~\ref{thm_semi} in Section~\ref{sec_semi}.

\section{Preliminaries}
\label{sec_prelim}
We begin this section by stating two initial boundary value problems that we will use throughout the remainder of this section. These are stated as follows.
\begin{equation}\label{eq_free}
	\begin{aligned}
		\begin{cases}
			\p_t u - \Delta u=0 
			&\text{on  $M=(0,T)\times \Omega$},
			\\
			u=f & \text{on $\Sigma=(0,T)\times \p \Omega$},
			\\
			u|_{t=0}=0 & \text{on $\Omega$},
		\end{cases}
	\end{aligned}
\end{equation}
and its adjoint equation
\begin{equation}\label{eq_free_adjoint}
	\begin{aligned}
		\begin{cases}
			\p_t u + \Delta u=0 
			&\text{on  $M=(0,T)\times \Omega$},
			\\
			u=h & \text{on $\Sigma=(0,T)\times \p \Omega$},
			\\
			u|_{t=T}=0 & \text{on $\Omega$}.
		\end{cases}
	\end{aligned}
\end{equation}
For the adjoint equation \eqref{eq_free_adjoint}, we note that given any 
$$ h \in \mathcal H^\star(\Sigma)=  \{ g\in H^{\frac{3}{4},\frac{3}{2}}(\Sigma)\,:\, g|_{t=T}=0  \},$$
there exists a unique solution $u \in H^{1,2}(M)$. Our aim in the remainder of this section is to first give an explicit derivation of the Fr\'{e}chet derivative $\mathcal S^\Gamma_q$ followed by a characterization of the injectivity of the Fr\'{e}chet derivative in terms of a completeness question for products of solutions to heat equations. To this end, let us consider $f\in \mathcal H(\Sigma)$ with $\supp f\subset (0,T]\times \Gamma$ and denote by $u_s$ the unique solution to \eqref{heat} with $q$ replaced by $sq$. Then, $u_s$ depends smoothly on $s$ taking values in $H^{1,2}(M)$ and therefore, 
$$ u_s = u_0 + s v + O(s^2) \quad \text{for all $|s|\leq 1$},$$
where $u_0$ solves \eqref{eq_free} with Dirichlet boundary data $f$ and $v$ is the unique solution to 
\begin{equation}\label{eq_free_source}
	\begin{aligned}
		\begin{cases}
			\p_t v - \Delta v= -qu_0 
			&\text{on  $M=(0,T)\times \Omega$},
			\\
			v=0 & \text{on $\Sigma=(0,T)\times \p \Omega$},
			\\
			v|_{t=0}=0 & \text{on $\Omega$}.
		\end{cases}
	\end{aligned}
\end{equation}
Therefore, we have 
\begin{equation}\label{Frechet_der} 
	\mathcal S^\Gamma_q(f) = \p_\nu v|_{(0,T)\times \Gamma} \quad \forall\, f\in \mathcal H(\Sigma)\quad \text{with}\quad \supp f \subset (0,T]\times \Gamma.
	\end{equation}
We have the following two lemmas about the linearized Dirichlet-to-Neumann map.
\begin{lemma}
	\label{lem_DN_map_Frechet}
	Let $q_1,q_2\in L^{\infty}(M)$. Given any 
	$$(f,h) \in C^{\infty}_0((0,T]\times \Gamma)\times C^{\infty}_0([0,T)\times \Gamma),$$  there holds:
	$$ \int_{(0,T)\times \Gamma} h\,(\mathcal S^\Gamma_{q_1}-\mathcal S^\Gamma_{q_2})(f)\,dt\,d\sigma = \int_{M} (q_1-q_2)\,w_1 \,w_2\,dt\,dx,$$ 
	where $w_1$ and $w_2$ are the unique solutions to \eqref{eq_free} and \eqref{eq_free_adjoint} with Dirichlet boundary datum $f$ and $h$ respectively.
\end{lemma}
\begin{proof}
In view of \eqref{Frechet_der} we have
$$ \mathcal S_{q_1}^\Gamma(f) = \p_\nu v_1|_{(0,T)\times \Gamma} \quad \text{and}\quad  \mathcal S_{q_2}^\Gamma(f) = \p_\nu v_2|_{(0,T)\times\Gamma},$$
where $v_j\in H^{1,2}(M)$, $j=1,2,$ is the unique solution to \eqref{eq_free_source} with $q=q_j$ and $u_0=w_1$. Let $v=v_1-v_2$ and note that 
$$(\mathcal S^\Gamma_{q_1}-\mathcal S^\Gamma_{q_2})(f)=\p_\nu v|_{(0,T)\times \Gamma}.$$
Note also that $v$ solves the following equation
\begin{equation*}
	\begin{aligned}
		\begin{cases}
			\p_t v - \Delta v= -(q_1-q_2)w_1 
			&\text{on  $M=(0,T)\times \Omega$},
			\\
			v=0 & \text{on $\Sigma=(0,T)\times \p \Omega$},
			\\
			v|_{t=0}=0 & \text{on $\Omega$}.
		\end{cases}
	\end{aligned}
\end{equation*}
Multiplying the above equation with $w_2$ and integrating by parts yields the claim.
\end{proof}

\begin{lemma}
	\label{lem_density}
	Let $\Gamma \subset \partial\Omega$ be a nonempty open set. Let $q_1,q_2 \in L^{\infty}(M)$ satisfy $q_1=q_2$ on $\mathcal D_\delta$ for some $\delta>0$. Then, 
	$$\mathcal S^\Gamma_{q_1}=\mathcal S^\Gamma_{q_2}\implies \int_\delta^{T-\delta}\int_\Omega (q_1-q_2)\, u_1\,u_2 \,dx\,dt =0,$$
	for any pair $u_1\in C^{\infty}([\delta,T]\times \overline\Omega)$ and $u_2 \in C^{\infty}([0,T-\delta]\times \overline\Omega)$ that respectively satisfy 
	$$ \p_t u_1 -\Delta u_1=0 \quad \text{on $[\delta,T]\times \Omega$} \quad \text{with} \quad \supp (u_1|_{[\delta,T]\times \p \Omega})\subset [\delta,T]\times \Gamma,$$
	and 
	$$ \p_t u_2 +\Delta u_2=0 \quad \text{on $(0,T-\delta)\times \Omega$} \quad \text{with} \quad \supp (u_2|_{[0,T-\delta]\times \p \Omega})\subset [0,T-\delta]\times \Gamma.$$
\end{lemma}

\begin{proof}
	Throughout this proof, for any pair $f\in C^{\infty}_0((0,T]\times \Gamma)$ and $h \in C^{\infty}_0([0,T)\times \Gamma)$ we write $v_f$ and $w_h$ to stand for the solutions to \eqref{eq_free} and \eqref{eq_free_adjoint} with Dirichlet boundary data $f$ and $h$ respectively. Since $\mathcal S^\Gamma_{q_1}=\mathcal S^\Gamma_{q_2}$ and $q_1=q_2$ on $\mathcal D_\delta$, it follows from Lemma~\ref{lem_DN_map_Frechet} that 
	\begin{equation}\label{density_full} \int_{\delta}^{T-\delta}\int_\Omega (q_1-q_2)v_f w_h\,dx\,dt=0.\end{equation}
	Let $f_0 \in C^{\infty}_0((0,T]\times \Gamma)$ and $h_0\in C^{\infty}_0([0,T)\times \Gamma)$ be chosen such that 
	$$ f_0|_{(\delta,T)\times \Gamma}= u_1|_{(\delta,T)\times \Gamma} \quad \text{and} \quad h_0|_{(0,T-\delta)\times \Gamma}= u_2|_{(0,T-\delta)\times \Gamma},$$
	and subsequently define $v_{f_0}$ and $w_{h_0}$ as above. By approximate controllability for the heat equation, there exists $\{f_k\}_{k=1}^{\infty} \subset C^{\infty}_0((0,\delta)\times \Gamma)$ and  $\{h_k\}_{k=1}^{\infty} \subset C^{\infty}_0((T-\delta,T)\times \Gamma)$ such that 
	$$ v_{f_k}|_{t=\delta} \to (u_1- v_{f_0})|_{t=\delta} \quad \text{ in $L^2(\Omega)$ as $k\to \infty$},$$ 
	and
	$$ w_{h_k}|_{t=T-\delta} \to (u_2- w_{h_0})|_{t=T-\delta} \quad \text{ in $L^2(\Omega)$ as $k\to \infty$}.$$
Fixing $v_k=(v_{f_k+f_0}-u_1)|_{(\delta,T)\times\Omega}$ and $w_k=(w_{h_k+h_0}-u_2)|_{(0,T-\delta)\times\Omega}$, we see that these functions solve the following problems
	\begin{equation*}
	\begin{aligned}
		\begin{cases}
			\p_t v_k - \Delta v_k= 0 
			&\text{on  $(\delta,T)\times \Omega$},
			\\
			v_k=f_0-u_1=0 & \text{on $(\delta,T)\times \p \Omega$},
			\\
			v_k|_{t=\delta}=v_{f_k}|_{t=\delta}-(u_1- v_{f_0})|_{t=\delta} & \text{on $\Omega$}.
		\end{cases}
	\end{aligned}
\end{equation*}
		\begin{equation*}
	\begin{aligned}
		\begin{cases}
			-\p_t w_k - \Delta w_k= 0 
			&\text{on  $(0,T-\delta)\times \Omega$},
			\\
			w_k=h_0-u_2=0 & \text{on $(0,T-\delta)\times \p \Omega$},
			\\
			w_k|_{t=T-\delta}=w_{h_k}|_{t=T-\delta}-(u_2- w_{h_0})|_{t=T-\delta} & \text{on $\Omega$}.
		\end{cases}
	\end{aligned}
\end{equation*}
By energy estimates for the heat equation, we deduce that there exists $C>0$, depending only on $\Omega$, $T$ and $\delta$, such that
$$\|v_k\|_{L^2(\delta,T;H^1(\Omega))}\leq C\|v_{f_k}|_{t=\delta}-(u_1- v_{f_0})|_{t=\delta}\|_{L^2(\Omega)},$$
$$\|w_k\|_{L^2(0,T-\delta;H^1(\Omega))}\leq C\|w_{h_k}|_{t=\delta}-(u_2- w_{h_0})|_{t=T-\delta}\|_{L^2(\Omega)}$$
and, sending $k\to\infty$, we obtain
	\begin{equation}\label{heat_sol_inf} \|v_{f_k+f_0}-u_1\|_{L^2((\delta,T)\times \Omega)} + \|w_{h_k+h_0}-u_2\|_{L^2((0,T-\delta)\times \Omega)} \to 0 \quad \text{as $k\to \infty$.}\end{equation}
	Observe in view of \eqref{density_full} that there holds:
	$$\int_{\delta}^{T-\delta}\int_\Omega (q_1-q_2)\,v_{f_0+f_k} \,w_{h_0+h_k}\,dx\,dt=0 \quad \forall\, k\in \N.$$ 
	The claim follows by letting $k$ approach infinity and using \eqref{heat_sol_inf}.
\end{proof}

\section{Local completeness result}
\label{sec_local}
The aim of this section is to prove a local version of Theorem~\ref{thm_density} that will be augmented later with a Runge approximation argument to obtain the global version. Precisely, in this section, we will prove the following theorem.
\begin{theorem}
	\label{thm_density_loc}
	Let $M=(0,T)\times \Omega$ where $T>0$ and $\Omega\subset \R^n$ is a connected bounded domain with smooth boundary. Let $\Gamma\subset \p\Omega$ be a nonempty open set and let $p\in \Gamma$. Let $q\in L^{\infty}(M)$ satisfy $q=0$ on $\mathcal D_\delta$ for some $\delta>0$. Suppose that 
	\begin{equation}\label{density_eq_loc} \int_\delta^{T-\delta}\int_\Omega q\, u_1\,u_2 \,dx\,dt =0,\end{equation}
	for any pair $u_1\in C^{\infty}([\delta,T]\times \overline\Omega)$ and $u_2 \in C^{\infty}([0,T-\delta]\times \overline\Omega)$ that respectively satisfy 
	$$ \p_t u_1 -\Delta u_1=0 \quad \text{on $(\delta,T)\times \Omega$} \quad \text{with} \quad \supp (u_1|_{[\delta,T]\times \p \Omega})\subset [\delta,T]\times \Gamma,$$
	and 
	$$ \p_t u_2 +\Delta u_2=0 \quad \text{on $(0,T-\delta)\times \Omega$} \quad \text{with} \quad \supp (u_2|_{[0,T-\delta]\times \p \Omega})\subset [0,T-\delta]\times \Gamma.$$
	Then, there is a neighborhood $\mathcal V$ of $p$ such that $q=0$ on $(0,T)\times (\mathcal V\cap \Omega)$.
\end{theorem}

\subsection{Preliminaries} For the benefit of the reader we start this section with fixing our notation that will stand for the entire section. We will also state some lemmas for future reference. The notation $B(x,r)$ stands for the closed ball of radius $r>0$ centred at the point $x\in \R^n$. We will generally work in a neighbourhood of the point $p\in \Gamma$ (see  Theorem~\ref{thm_density_loc}) and as such we introduce a suitable local coordinate system near this point. To this end, let $x_0= p+\varepsilon_0\nu(p)$ where $\nu(p)$ is the exterior unit normal vector to $\p \Omega$ at $p$. We write $\Pi_{x_0}$ for the  hyperplane that is normal to $\nu(p)$ and passes through the point $x_0$ (see Figure \eqref{fig1} for more details). Since $\p \Omega $ is smooth, there is $\varepsilon_0\in(0,1)$ small such that: 
\begin{itemize}
	\item [{\bf (C1)}]{The ball $B(x_0,\varepsilon_0)$ only intersects the boundary $\p \Omega$ at the point $p$.}
	\item[{\bf (C2)}]{For each $r\in [\varepsilon_0,2\varepsilon_0]$, $B(x_0,r) \cap \p\Omega$ is a connected nonempty subset of $\Gamma$.}
	\item[{\bf (C3)}]{$\Omega \cap B(x_0,2\varepsilon_0) \cap \Pi_{x_0}=\emptyset$.}
\end{itemize}
 We fix $\varepsilon_0\in (0,1)$ only depending on $\Omega$, $\Gamma$ and $p$ so that the above three conditions are satisfied. Subsequently, given each $r\in[\varepsilon_0,2\varepsilon_0]$ denote by $U_r$ the unique hemisphere on the boundary of $B(x_0,r)$ that both has its great circle on $\Pi_{x_0}$ and that the hemisphere also intersects $\p \Omega$. We now consider the region 
\begin{equation}\label{U_def_0} 
	U = \bigcup_{r=\varepsilon_0}^{2\varepsilon_0} U_r,
\end{equation}
and define coordinates on $U$ via polar coordinates centred at the point $x_0$. In this way we can say that in polar coordinates 
\begin{equation}
	\label{U_def}
U = [\varepsilon_0,2\varepsilon_0]\times \mathbb S^{n-1}_+ \quad g = (dr)^2+ r^2 g_{\mathbb S^{n-1}_+} \end{equation}
where $(\mathbb S^{n-1}_+,g_{\mathbb S^{n-1}_+})$ is the unit $(n-1)$-dimensional hemisphere embedded $\R^n$ with its natural induced metric. For example, in $\R^2$, we have $\mathbb S^1_+= [0,\pi]$ with $g_{\mathbb S^1_+}= (d\theta)^2$. From now on, we will use the coordinates $(r,\theta)=(r,\theta^1,\ldots,\theta^{n-1})$ described above on $U$.  
\begin{figure}[!ht]
  \centering
  \includegraphics[width=0.6\textwidth]{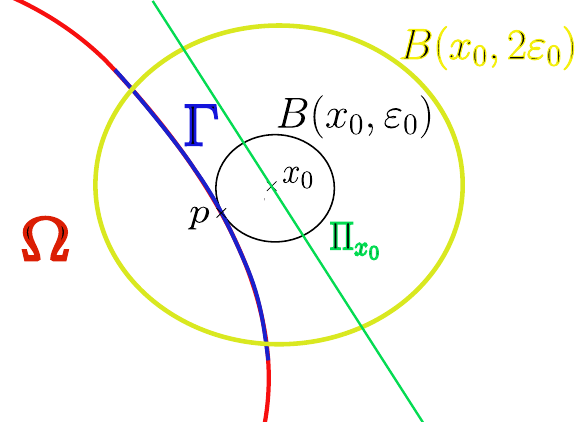}
  \caption{Figure for $p$, $x_0$ $B(x_0,\varepsilon_0)$ and $\Pi_{x_0}$  }
  \label{fig1}
\end{figure}
Observe that in view of (C1) there is a positive $\varepsilon_1\in (0,\frac{\varepsilon_0}{4})$ such that
\begin{equation}
	\label{distant_formula}
	2\varepsilon_0\geq\textrm{dist}(x,x_0)\geq \varepsilon_0+\varepsilon_1 \quad \forall\, x\in \overline\Omega \cap B(p,\frac{\varepsilon_0}{2})\setminus B(p,\frac{\varepsilon_0}{4}).
\end{equation}
Next, we define a family of spherical functions that will be employed later.	
\begin{definition}\label{Y_def}
Let $\sigma \in [0,1]$. For $n=2$, we define
$$ Y_\sigma(\theta)=e^{\sigma \theta} \quad \theta \in [0,\pi].$$
For $n\geq 3$, and any $\omega \in C^{\infty}(\p \mathbb S_+^{n-1})$ that additionally satisfies 
\begin{equation}\label{def_omega}
	\|\omega\|_{H^{\frac{n+3}{2}}(\p \mathbb S^{n-1}_+)} \leq 1,\end{equation}
we define the function $Y_{\sigma,\omega} \in C^{\infty}(\overline{\mathbb S_+^{n-1}})$ to be the unique solution to 
	\begin{equation}\label{spherical_harmonics}
	\begin{aligned}
		\begin{cases}
			-\Delta_{\mathbb S^{n-1}_+} Y_{\sigma,\omega} +\sigma^2 Y_{\sigma,\omega}=0 
			&\text{on  $\mathbb S^{n-1}_+$},
			\\
			Y_{\sigma,\omega}=\omega & \text{on $\p \mathbb S_+^{n-1}$}.
		\end{cases}
	\end{aligned}
\end{equation}
\end{definition}
In what follows, we abuse notation slightly and always write $Y_{\sigma,\omega}$ independent of the dimension $n=2$ or $n>2$. It should however be understood implicitly that in the case $n=2$ the role of the function $\omega$ is redundant. Let us remark also that by Sobolev embedding theorem and elliptic estimates there holds:
\begin{equation}
	\label{Y_est}
	\|Y_{\sigma,\omega}\|_{W^{2,\infty}( \mathbb S_+^{n-1})} \leq C_n, 
\end{equation}
for some universal constant $C_n>0$ that depends only on $n$. We will also need to define a fixed cut-off function near the point $p$ that will be employed in the proof of Theorem~\ref{thm_density_loc}.
\begin{definition}\label{def_chi}
We define $\chi \in C^{\infty}(\R^n;[0,1])$ such that $\chi=1$ inside the ball $B(p,\frac{\varepsilon_0}{4})$ and $\chi=0$ outside the ball $B(p,\frac{\varepsilon_0}{2})$.
\end{definition}
Let us observe that supp$(\chi)\subset U$.
Next, for future reference, let us introduce a family of smooth functions as follows.
\begin{definition}\label{def_amplitudes}
	Let $\sigma \in [0,1]$. We define $\{a_k\}_{k=0}^{\infty} \subset C^{\infty}((0,\infty))$ via
	$$ a_k(r) = c_k\,r^{-\frac{n-1}{2}-k},$$
	where $c_0=1$ and for each $k\in \N$, there holds
	\begin{equation}\label{def_c} 
		c_k = -\frac{1}{2k}\left(k^2-k+\sigma^2 -\frac{(n-1)(n-3)}{4}\right)\,c_{k-1}.\end{equation} 
\end{definition}
We have the following trivial lemma.
\begin{lemma}
	\label{lem_amplitudes}
	The sequence of functions $\{a_k(r)\}_{k=1}^{\infty}\subset C^{\infty}((0,\infty))$ described above satisfy the following recursive ODEs:
	$$ 2\frac{da_k}{dr} +\frac{n-1}{r} a_{k} - r^{-(n-1)} \frac{d}{dr}(r^{n-1}\frac{da_{k-1}}{dr}) - \sigma^2 r^{-2}a_{k-1} =0 \quad \forall\, k\in \N.$$
	Moreover, given any $k\in \N$, there holds:
	\begin{equation}
		\label{c_bounds}
		|c_k| \leq C2^{-k}\,(k+1)!,
	\end{equation}  
	with $C>0$ a constant depending only on $n$.
\end{lemma}

\subsection{Spherical quasimodes for heat equation}

We start by introducing a parameter 
\begin{equation}\label{tau_range}\tau>\min\{n,\frac{64e}{\varepsilon_0}\} \quad \text{and} \quad N=\lfloor \frac{\varepsilon_0\tau}{32e}\rfloor.\end{equation}
Our aim in this section is to introduce a family of solutions to heat equations that depend on the asymptotic parameter $\tau$ that will eventually approach infinity. These will be the family of solutions that we will use in proving Theorem~\ref{thm_density_loc}. More precisely, we begin by fixing
$$ \sigma_1,\sigma_2 \in [0,1] \quad \text{and $\omega_1,\omega_2 \in C^{\infty}(\p \mathbb S_+^{n-1})$ satisfying \eqref{def_omega}},$$ 
and subsequently consider a smooth solution $u_{\tau}^{(1)}\in C^{\infty}([\delta,T]\times \overline{\Omega})$ to the equation
\begin{equation}\label{eq_free_delta}
			\p_t u_{\tau}^{(1)} - \Delta u_{\tau}^{(1)}=0 
			\quad \text{on  $(\delta,T)\times \Omega$}, \quad \supp(u_{\tau}^{(1)}|_{[\delta,T]\times\p\Omega})\subset [\delta,T]\times \Gamma,
\end{equation}
of the form 
\begin{equation}\label{u_1_tau} u_{\tau}^{(1)}(t,x) = e^{\tau^2t}(\mathcal U_{\tau}^{(1)}(x) \chi(x)+ R_{\tau}^{(1)}(t,x)), \end{equation}
as well as a smooth solution $u_{\tau}^{(2)}\in C^{\infty}([0,T-\delta]\times \overline{\Omega})$ to the equation
\begin{equation}\label{eq_free_adjoint_delta}
		-\p_t u_{\tau}^{(2)} - \Delta u_{\tau}^{(2)}=0 
		\quad \text{on  $(0,T-\delta)\times \Omega$}\quad \supp(u_{\tau}^{(2)}|_{[0,T-\delta]\times\p\Omega})\subset [0,T-\delta]\times \Gamma,
	\end{equation}
of the form
\begin{equation}\label{u_2_tau} u_{\tau}^{(2)}(t,x) = e^{-\tau^2t}(\mathcal U_{\tau}^{(2)}(x) \chi(x)+ R_{\tau}^{(2)}(t,x)).\end{equation}
Here, $\chi$ is as in Definition~\ref{def_chi} and for each $j=1,2,$ the function $\mathcal U_{\tau}^{(j)} (x)$ is defined in the region $U$ (see \eqref{U_def}) and is given in polar coordinates $(r,\theta)$ centred at the point $x_0$ via the expression 
\begin{equation}\label{principal_amp} 
	\mathcal U_{\tau}^{(j)}(r,\theta) =e^{-\tau r} (\underbrace{\sum\limits_{k=0}^{N}a^{(j)}_k(r)\tau^{-k}}_{A^{(j)}_\tau(r)})Y_{\sigma_j,\omega_j}(\theta), \quad r\in [\varepsilon_0,2\varepsilon_0] \quad \theta \in \mathbb S_+^{n-1},\end{equation}
where the sequence of functions $$a_k^{(j)}=c^{(j)}_k\,r^{-\frac{n-1}{2}-k} \quad k=0,1\ldots,N,$$ 
are given as in Definition~\ref{def_amplitudes} and $c_k^{(j)}$ are defined by \eqref{def_c} with $\sigma=\sigma_j$ and finally we recall that $Y_{\sigma_j,\omega_j}$ is defined as in \eqref{spherical_harmonics} with $\sigma=\sigma_j \in [0,1]$ and $\omega=\omega_j\in C^{\infty}(\p \mathbb S_+^{n-1})$ as in \eqref{def_omega}. We recall that in the case $n=2$, the role of $\omega_j$ is redundant. In general, we remark to the reader that the superscript $(j)$ with $j=1,2,$ always depicts the dependence of solutions on the pair $(\sigma_j,\omega_j)$.  We have the following estimate for the amplitude term $A_\tau^{(j)}$.
\begin{lemma}
	\label{lem_amp_est}
	For $j=1,2$ and any $\tau>0$ satisfying \eqref{tau_range} there holds:
	$$ \| A^{(j)}_\tau(r)-a^{(j)}_0(r)\|_{W^{2,\infty}((\varepsilon_0,2\varepsilon_0))} \leq C_{\varepsilon_0}\,\tau^{-1},$$
for some constant $C_{\varepsilon_0}>0$ that is independent of $\tau$.
\end{lemma}

\begin{proof}
	First, note that 
	$$ \|a_k^{(j)}(r)\|_{W^{2,\infty}((\varepsilon_0,2\varepsilon_0))} \leq C(k+1)!\,2^{-k}\, \|r^{-\frac{n-1}{2}-k}\|_{W^{2,\infty}((\varepsilon_0,2\varepsilon_0))}\leq D_{\varepsilon_0}\, \left(\frac{k+1}{\varepsilon_0}\right)^k,$$
	for some constant $D_{\varepsilon_0}>0$ that depends only on $\varepsilon_0$ and $n$. Therefore,  we write
	$$\begin{aligned} \|A_\tau^{(j)}(r)-a_0^{(j)}(r)\|_{W^{2,\infty}((\varepsilon_0,2\varepsilon_0))}&\leq \tau^{-1}\,\sum_{k=1}^{N} \|a_k^{(j)}(r)\|_{W^{2,\infty}((\varepsilon_0,2\varepsilon_0))} \tau^{-k+1}\\
		&\leq \tau^{-1}\,\sum_{k=1}^{N} D_{\varepsilon_0}\,\frac{k+1}{\varepsilon_0}\, \left(\frac{k+1}{\tau\,\varepsilon_0}\right)^{k-1} .
	\end{aligned}$$
	Recalling that $N$ is defined by \eqref{tau_range}, for all $k=1,\ldots,N$, we obtain
	$$\left(\frac{k+1}{\tau\,\varepsilon_0}\right)^{k-1}\leq \left(\frac{N+1}{\tau\,\varepsilon_0}\right)^{k-1}\leq \left(\frac{\frac{\tau\varepsilon_0}{32e}+1}{\tau\,\varepsilon_0}\right)^{k-1}\leq \left(\frac{1}{2}\right)^{k-1}$$
	and it follows that
	$$\|A_\tau^{(j)}(r)-a_0^{(j)}(r)\|_{W^{2,\infty}((\varepsilon_0,2\varepsilon_0))}\leq \tau^{-1}\,\underbrace{\frac{D_{\varepsilon_0}}{\varepsilon_0} \sum_{k=1}^{\infty}(k+1)\left(\frac{1}{2}\right)^{k-1}}_{C_{\varepsilon_0}}.$$	
\end{proof}
With the ansatz \eqref{u_1_tau}--\eqref{u_2_tau}, we first aim to demonstrate that the principal part approximately satisfies the heat equations \eqref{eq_free_delta}--\eqref{eq_free_adjoint_delta} respectively. More precisely, let us start by writing 
\begin{multline} \label{conjug_1}
	(\p_t -\Delta) (e^{\tau^2t-\tau r}A^{(1)}_\tau(r)Y_{\sigma_1,\omega_1}(\theta))=\\
	 e^{\tau^2t-\tau r}\left(2\tau \p_r A^{(1)}_\tau +\tau \frac{n-1}{r}A^{(1)}_\tau - r^{-(n-1)}\p_r(r^{n-1}\p_r A^{(1)}_\tau)-\frac{\sigma_1^2}{r^2}A_\tau^{(1)}\right) Y_{\sigma_1,\omega_1}(\theta),
\end{multline}
and 
\begin{multline} \label{conjug_1_alt}
	(-\p_t -\Delta) (e^{-\tau^2t-\tau r}A^{(2)}_\tau(r)Y_{\sigma_2,\omega_2}(\theta))=\\
	e^{-\tau^2t-\tau r}\left(2\tau \p_r A^{(2)}_\tau +\tau \frac{n-1}{r}A^{(2)}_\tau - r^{-(n-1)}\p_r(r^{n-1}\p_r A^{(2)}_\tau)-\frac{\sigma_2^2}{r^2}A_\tau^{(2)}\right) Y_{\sigma_2,\omega_2}(\theta),
\end{multline}
 In view of Lemma~\ref{lem_amplitudes} there holds
 \begin{multline} 
 	\label{conj_2}
 	(\p_t -\Delta) (e^{\tau^2t-\tau r}A^{(1)}_\tau(r)Y_{\sigma_1,\omega_1}(\theta))=\\
 	-c_N^{(1)}\,\tau^{-N}\,\left((N-\frac{n-3}{2})(N+\frac{n-1}{2})+\sigma_1^2 \right)\,e^{\tau^2t-\tau r} r^{-\frac{n-1}{2}-N-2} \,Y_{\sigma_1,\omega_1}(\theta)
 \end{multline}
and 
 \begin{multline} 
	\label{conj_3}
	(-\p_t -\Delta) (e^{-\tau^2t-\tau r}A^{(2)}_\tau(r)Y_{\sigma_2,\omega_2}(\theta))=\\
	-c_N^{(2)}\,\tau^{-N}\,\left((N-\frac{n-3}{2})(N+\frac{n-1}{2})+\sigma_2^2 \right)\,e^{-\tau^2t-\tau r} r^{-\frac{n-1}{2}-N-2} \,Y_{\sigma_2,\omega_2}(\theta),
\end{multline}
where $c_N^{(j)}$, $j=1,2,$ are as in Definition~\ref{def_amplitudes} with $\sigma=\sigma_j$. 

We construct the remainder terms $R_{\tau}^{(1)}$ and $R_{\tau}^{(2)}$ by solving the following initial boundary value problems:
\begin{equation}\label{remainder_1}
	\begin{aligned}
		\begin{cases}
			(\p_t-\Delta) (e^{\tau^2 t}R_{\tau}^{(1)})=- (\p_t -\Delta)(e^{\tau^2t}\mathcal U_{\tau}^{(1)}(x)\chi(x)) 
			&\text{on  $(\delta,T)\times \Omega$},
			\\
			R_{\tau}^{(1)}=0 & \text{on $(\delta,T)\times \p \Omega$},
			\\
			R_{\tau}^{(1)}|_{t=\delta}=0 & \text{on $\Omega$}.
		\end{cases}
	\end{aligned}
\end{equation}
and
\begin{equation}\label{remainder_2}
	\begin{aligned}
		\begin{cases}
			(-\p_t-\Delta) (e^{-\tau^2 t}R_{\tau}^{(2)})=(\p_t+\Delta)(e^{-\tau^2t}\mathcal U^{(2)}_{\tau}(x)\chi(x)) 
			&\text{on  $(0,T-\delta)\times \Omega$},
			\\
			R_{\tau}^{(2)}=0 & \text{on $(0,T-\delta)\times \p \Omega$},
			\\
			R_{\tau}^{(2)}|_{t=T-\delta}=0 & \text{on $\Omega$}.
		\end{cases}
	\end{aligned}
\end{equation}
Taking \eqref{conj_2} into account, we can rewrite equations \eqref{remainder_1}--\eqref{remainder_2} as follows:
\begin{equation}\label{remainder_12}
	\begin{aligned}
		\begin{cases}
			e^{-\tau^2 t}(\p_t-\Delta) (e^{\tau^2 t}R_{\tau}^{(1)})=F_{\tau}^{(1)}(x)+G_{\tau}^{(1)}(x),
			&\text{on  $(\delta,T)\times \Omega$},
			\\
			R_{\tau}^{(1)}=0 & \text{on $(\delta,T)\times \p \Omega$},
			\\
			R_{\tau}^{(1)}|_{t=\delta}=0 & \text{on $\Omega$},
		\end{cases}
	\end{aligned}
\end{equation}
and
\begin{equation}\label{remainder_22}
	\begin{aligned}
		\begin{cases}
			e^{\tau^2 t}(-\p_t-\Delta) (e^{-\tau^2 t}R_{\tau}^{(2)})=F_{\tau}^{(2)}(x)+G_{\tau}^{(2)}(x),
			&\text{on  $(0,T-\delta)\times \Omega$},
			\\
			R_{\tau}^{(2)}=0 & \text{on $(0,T-\delta)\times \p \Omega$},
			\\
			R_{\tau}^{(2)}|_{t=T-\delta}=0 & \text{on $\Omega$},
		\end{cases}
	\end{aligned}
\end{equation}
where for $j=1,2,$ we have
\begin{equation}
	\label{F_1}
	F_{\tau}^{(j)}=c_N^{(j)}\,\tau^{-N}\,\left((N-\frac{n-3}{2})(N+\frac{n-1}{2})+\sigma_j^2 \right)\,e^{-\tau r} r^{-\frac{n-1}{2}-N-2} \,Y_{\sigma_j,\omega_j}\,\chi
\end{equation}
and
\begin{equation}
	\label{F_2}
	G_{\tau}^{(j)}=[\Delta,\chi] \,\mathcal U^{(j)}_{\tau}
\end{equation}
We have the following lemma. We recall that $\varepsilon_1$ is as defined in \eqref{distant_formula}.
\begin{lemma}
	\label{lem_F_est}
There exists $\varepsilon_2\in (0,\frac{\varepsilon_1}{32e})$ such that given any $\tau$ and $N$ satisfying \eqref{tau_range} there holds:
\begin{equation}
	\label{F_est}
	\|F_{\tau}^{(j)}\|_{L^2(\Omega)}+	\|G_{\tau}^{(j)}\|_{L^2(\Omega)} \leq  C_{\varepsilon_0}\, e^{-(\varepsilon_0+2\varepsilon_2)\,\tau}, \quad j=1,2,
\end{equation}
where $C_{\varepsilon_0}>0$ is independent of $\tau$.
\end{lemma}
\begin{proof}
	We start with $F_{\tau}^{(j)}$ which is supported in the region $U$. Together with \eqref{Y_est} and \eqref{c_bounds}, we have
	$$ \|F_{\tau}^{(j)}\|_{L^2(\Omega)} \leq C\,(N+1)!\,(2\tau)^{-N}\,N^2\, \| e^{-\tau r}r^{-\frac{n-1}{2}-N-2}\|_{L^2((\varepsilon_0,2\varepsilon_0))},$$
for some constant $C>0$ that only depends on $n$. Recalling that $\varepsilon_0<1$, $N!\leq N^N$ and that $N$ is given by \eqref{tau_range}, it follows that
$$\begin{aligned} \|F_{\tau}^{(j)}\|_{L^2(\Omega)}&\leq C\,(N+1)!\,(2\tau)^{-N}\,N^2\, \varepsilon_0^{\frac{1}{2}}\| e^{-\tau r}r^{-\frac{n-1}{2}-N-2}\|_{L^\infty((\varepsilon_0,2\varepsilon_0))} \\
	&\leq C\,\varepsilon_0^{-\frac{n+2}{2}}\,\left(\frac{N}{2\tau\varepsilon_0}\right)^{N}\,(N+1)^3\, e^{-\tau \varepsilon_0}\\
	&\leq C\,\varepsilon_0^{-\frac{n+2}{2}}\,e^{-(N+1)}\,(N+1)^3\,2^{-N}\,e^{-\tau \varepsilon_0}\\
	&\leq  C\,\varepsilon_0^{-\frac{n+2}{2}}\,e^{-(\varepsilon_0+\frac{\varepsilon_0}{32e})\tau} ,\end{aligned}$$
where we recall that $N+1\geq\frac{\varepsilon_0 \tau}{32e}$ and we denote by $C>0$ a generic constant that may change from line to line depending on $n$. For $G_{\tau}^{(j)}$ we recall that the commutator is supported only in the region $B(p,\frac{\varepsilon_0}{2})\setminus B(p,\frac{\varepsilon_0}{4})$ and that in view of \eqref{distant_formula} we have $r\geq \varepsilon_0+\varepsilon_1$ in this region. Therefore,
$$ \|G_{\tau}^{(j)}\|_{L^2(\Omega)} \leq C_{\varepsilon_0}\, \|e^{-\tau r}A_\tau^{(j)}(r) \|_{H^{1}((\varepsilon_0+\varepsilon_1,2\varepsilon_0))},$$
for some constant $C_{\varepsilon_0}>0$ that only depends on $n$ and $\varepsilon_0$. Together with Lemma~\ref{lem_amp_est} we deduce that
$$  \|G_{\tau}^{(j)}\|_{L^2(\Omega)}  \leq C_{\varepsilon_0}\, \tau\, e^{-\tau( \varepsilon_0+\varepsilon_1)},$$
for some constant $C_{\varepsilon_0}>0$ that only depends on $n$ and $\varepsilon_0$. The claim now follows from the above estimates for $F_{\tau}^{(j)}$ and $G_{\tau}^{(j)}$, $j=1,2$, by choosing $\varepsilon_2 \in (0,\frac{\varepsilon_1}{32e})$ sufficiently small, independent of $\tau$.
\end{proof}
We are ready to give estimates on the remainder terms.
\begin{lemma}
	\label{lem_remainder_est}
	Given any $\tau, N$ satisfying \eqref{tau_range}, there holds:
	$$ \|R_{\tau}^{(1)}\|_{L^2((\delta,T)\times\Omega)}\leq \sqrt{T}\,C_{\varepsilon_0}\,e^{-(\varepsilon_0+2\varepsilon_2)\tau} \quad \text{and}\quad \|R_{\tau}^{(2)}\|_{L^2((0,T-\delta)\times\Omega)}\leq \sqrt{T}\,C_{\varepsilon_0}\,e^{-(\varepsilon_0+2\varepsilon_2)\tau},$$
	where $\varepsilon_2\in (0,\frac{\varepsilon_1}{32e})$ is as given by Lemma~\ref{lem_F_est} and $C_{\varepsilon_0}>0$ is independent of $\tau$.
\end{lemma}

\begin{proof}
	We write the proof for $R_{\tau}^{(1)}$ only as the proof for $R_{\tau}^{(2)}$ follows analogously. We start by noting that $R_{\tau}^{(1)}$ satisfies:
	$$ (\p_t-\Delta) R_{\tau}^{(1)}(t,x) + \tau^2\, R_{\tau}^{(1)}(t,x) = H_\tau(x):= F_{\tau}^{(1)}(x)+G_{\tau}^{(1)}(x) \quad \text{on $(\delta,T)\times \Omega$}.$$
	We multiply the above equation with the complex conjugate $\overline{R_{\tau}^{(1)}}$, integrate by parts over $(\delta,T)\times \Omega$ and take the real part of the expression as follows
	\begin{multline*}
		\textrm{Re} \int_{\delta}^T\int_\Omega H_\tau \overline{R_{\tau}^{(1)}} \,dx\,dt = 	\textrm{Re} \int_{\delta}^T\int_\Omega \overline{R_{\tau}^{(1)}} (\p_t R_{\tau}^{(1)}-\Delta R_{\tau}^{(1)} + \tau^2\, R_{\tau}^{(1)}) \,dx\,dt  \\
		\geq \int_{\delta}^T\int_\Omega (|\nabla R_{\tau}^{(1)}|^2 + \tau^2 |R_{\tau}^{(1)}|^2)\,dx\,dt\geq \tau^2  \int_{\delta}^T\int_\Omega |R_{\tau}^{(1)}|^2\,dx\,dt
	\end{multline*}
 Applying Cauchy--Schwarz inequality, it follows that
 $$ 	\int_{\delta}^T\int_\Omega (|H_\tau|^2 +|R_{\tau}^{(1)}|^2) \,dx\,dt \geq 2\tau^2  \int_{\delta}^T\int_\Omega |R_{\tau}^{(1)}|^2\,dx\,dt \geq 2 \int_{\delta}^T\int_\Omega |R_{\tau}^{(1)}|^2\,dx\,dt.$$
 The claim follows immediately thanks to \eqref{F_est} for $j=1$.
\end{proof}

\subsection{Proof of Theorem~\ref{thm_density_loc}} We have completed the quantitative description of the spherical quasimodes. Our goal now is to prove Theorem~\ref{thm_density_loc} by applying the spherical quasimode solutions as test functions into \eqref{density_eq_loc} and perform an asymptotic analysis as $\tau \to \infty$. To this end, let us start by choosing $$\lambda \in [0,1] \quad \text{and} \quad \tau>1+\min\{n,\frac{64e}{\varepsilon_0}\}.$$ 
Subsequently, we define the parameters
\begin{equation}\label{tau_1} \tau_1 = \tau + \frac{\lambda}{\tau} \quad \text{and} \quad N_1=N= \lfloor \frac{\varepsilon_0 \tau}{32e}\rfloor,\end{equation}
and 
\begin{equation}\label{tau_2} \tau_2 = \tau - \frac{\lambda}{\tau} \quad \text{and} \quad N_2=N= \lfloor \frac{\varepsilon_0 \tau}{32e}\rfloor.\end{equation}
In what follows, we fix $\lambda$ and let $\tau\to \infty$. Throughout the remainder of this section, the notation $C_{\varepsilon_0}$ denotes a generic constant that depends only on $\varepsilon_0$ (and also implicitly on the domain $\Omega$ and $T$). We remark that if a constant depends on $\varepsilon_1$ and $\varepsilon_2$ it may still be viewed as depending only on $\varepsilon_0$ as $\varepsilon_1$ and $\varepsilon_2$ depend on $\varepsilon_0$. We consider the smooth functions 
$$ u_1(t,x) = u_{\tau_1}^{(1)}(t,x)= e^{\tau_1^2t}(\mathcal U_{\tau_1}^{(1)}(x) \chi(x)+ R_{\tau_1}^{(1)}(t,x)) \quad \text{on $[\delta,T]\times \overline\Omega$},$$
and 
$$  u_2(t,x)= u_{\tau_2}^{(2)}(t,x)= e^{-\tau_2^2t}(\mathcal U^{(2)}_{\tau_2}(x) \chi(x)+ R_{\tau_2}^{(2)}(t,x)) \quad \text{on $[0,T-\delta]\times \overline\Omega$},$$
and plug these into equation \eqref{density_eq_loc}. We recall that for $j=1,2$, $\mathcal U_{\tau_j}^{(j)}$ is as in \eqref{principal_amp} with $\tau=\tau_j$ and $N=N_j$ (see \eqref{tau_1}-\eqref{tau_2}), that is to say, 
\begin{equation*}\label{principal_amp_j} 
	\mathcal U_{\tau_j}^{(j)}(r,\theta) =e^{-\tau_j r} (\underbrace{\sum\limits_{k=0}^{N}a_k^{(j)}(r)\tau_j^{-k}}_{A^{(j)}_{\tau_j}(r)})Y_{\sigma_j,\omega_j}(\theta), \quad r\in [\varepsilon_0,2\varepsilon_0] \quad \theta \in \mathbb S_+^{n-1},\end{equation*}
We also introduce
$$ q:= q_1-q_2 \quad \text{on $M$}.$$
We write
\begin{multline}\label{start_eq_den}
0=\int_\delta^{T-\delta}\int_\Omega q\, u_1\,u_2 \,dx\,dt\\
= \int_\delta^{T-\delta}\int_\Omega q\,e^{4\lambda t}(\mathcal U_{\tau_1}^{(1)}(x) \chi(x)+ R_{\tau_1}^{(1)}(t,x))\,(\mathcal U^{(2)}_{\tau_2}(x) \chi(x)+ R_{\tau_2}^{(2)}(t,x))\,dx\,dt\\
= \underbrace{\int_\delta^{T-\delta}\int_\Omega q\,e^{4\lambda t}\mathcal U_{\tau_1}^{(1)}(x) \mathcal U_{\tau_2}^{(2)}(x) \chi(x)^2\,dx\,dt}_{\textrm{I}}+ \underbrace{\int_\delta^{T-\delta}\int_\Omega q\,e^{4\lambda t}\mathcal U^{(1)}_{\tau_1}(x) R_{\tau_2}^{(2)}(t,x) \chi(x)\,dx\,dt}_{\textrm{II}}\\
+\underbrace{\int_\delta^{T-\delta}\int_\Omega q\,e^{4\lambda t}\mathcal U^{(2)}_{\tau_2}(x) R_{\tau_1}^{(1)}(t,x) \chi(x)\,dx\,dt}_{\textrm{III}}+ \underbrace{\int_\delta^{T-\delta}\int_\Omega q\,e^{4\lambda t}R_{\tau_1}^{(1)}(t,x) R_{\tau_2}^{(2)}(t,x)\,dx\,dt}_{\textrm{IV}}
\end{multline}
Applying Lemma~\ref{lem_amp_est} and Lemma~\ref{lem_remainder_est} it is straightforward to see that
\begin{equation}\label{II-III-IV}
|\textrm{II}| + |\textrm{III}| + |\textrm{IV}|  \leq C_{\varepsilon_0}\, e^{-(2\varepsilon_0+2\varepsilon_2)\tau},\end{equation}
for some constant $C_{\varepsilon_0}$ (depending only on $\varepsilon_0$, $\|q\|_{L^{\infty}(M)}$) and all $\tau>1+\min\{n,\frac{64e}{\varepsilon_0}\} $.
We turn now to term I which will take the bulk of our asymptotic analysis. Note that the integrand in the term I is supported near the point $p$ and thus we may turn to polar coordinates so that $(t,x)$ changes to $(t,r,\theta)$. We also extend $q$ to all of $(\delta,T-\delta)\times U$ (see \eqref{U_def}) by setting it to be zero outside of $M$. We can rewrite I as follows
$$\textrm{I} = \int_\delta^{T-\delta}\int_{\mathbb S_+^{n-1}}\int_{\varepsilon_0}^{2\varepsilon_0}  qe^{4\lambda t} e^{-2\tau r}\,A^{(1)}_{\tau_1}\,A^{(2)}_{\tau_2} Y_{\sigma_1,\omega_1}\,Y_{\sigma_2,\omega_2}\,\chi^2\,dt\,r^{n-1}\,dr\,dV_{\mathbb S_+^{n-1}}.$$
Recall from \eqref{start_eq_den}--\eqref{II-III-IV} that there holds 
$$ |\textrm{I}| \leq C_{\varepsilon_0}\, e^{-(2\varepsilon_0+2\varepsilon_2)\tau}.$$
Using \eqref{distant_formula} together with Definition~\ref{def_chi} we deduce that the contribution to the term I that comes from $B(p,\frac{\varepsilon_0}{2})\setminus B(p,\frac{\varepsilon_0}{4})$ is negligible since
$$ r \geq \varepsilon_0 + \varepsilon_1 \quad \text{on $B(p,\frac{\varepsilon_0}{2})\cap \overline\Omega\setminus B(p,\frac{\varepsilon_0}{4})$}.$$    This means that we may effectively assume $\chi=1$ (after changing $C_{\varepsilon_0}$ to a new constant) in the previous integral without losing any decay in $\tau$, that is to say,
\begin{multline*}
	\left|\int_\delta^{T-\delta}\,\int_{\mathbb S_+^{n-1}}\int_{\varepsilon_0}^{2\varepsilon_0}  q\,e^{4\lambda t} e^{-2\tau r}\,A_{\tau_1}^{(1)}(r)\,A_{\tau_2}^{(2)}(r) Y_{\sigma_1,\omega_1}(\theta)\,Y_{\sigma_2,\omega_2}(\theta)\,dt\,r^{n-1}\,dr\,dV_{\mathbb S_+^{n-1}}\right|\\
	\leq C_{\varepsilon_0}\, e^{-(2\varepsilon_0+2\varepsilon_2)\tau}.
\end{multline*}
Let us now define $Q \in L^{\infty}((\varepsilon_0,2\varepsilon_0))$ by 
\begin{equation}
	\label{def_Q}
	Q(r)= \int_\delta^{T-\delta}\int_{\mathbb S_+^{n-1}} q(t,r,\theta)\,e^{4\lambda t} Y_{\sigma_1,\omega_1}\,Y_{\sigma_2,\omega_2}\,dt\,dV_{\mathbb S_+^{n-1}} \quad \text{for a.e. $r\in (\varepsilon_0,2\varepsilon_0)$}.
\end{equation}
With this new definition, there holds:
\begin{equation}\label{Q_est}
	|\int_{\varepsilon_0}^{2\varepsilon_0} Q(r) \,e^{-2\tau r}\,A_{\tau_1}^{(1)}(r)\,A_{\tau_2}^{(2)}(r)\,r^{n-1}\,dr| \leq C_{\varepsilon_0}\, e^{-(2\varepsilon_0+2\varepsilon_2)\tau}.
	\end{equation}

Before proceeding any further with the asymptotic analysis of the above term as $\tau \to \infty$, we will need to study the term $A_{\tau_1}^{(1)}(r)A^{(2)}_{\tau_2}(r)$ in more detail. We start with defining the sequences $\{b_k\}_{k=0}^{\infty},\{d_k\}_{k=0}^{\infty},\{e_k\}_{k=0}^{\infty}\subset C^{\infty}([\varepsilon_0,2\varepsilon_0])$ as follows,
\begin{equation}
	\label{d_e_seq}
\begin{aligned}
	d_0(r)&=e_0(r)=a_0(r)=r^{-\frac{n-1}{2}}.\\
	d_k(r)&=\sum_{\substack{j=1\\\text{$k-j$ is even}}}^ka_j^{(1)}(r)\,(-\lambda)^{\frac{k-j}{2}} {{\frac{k+j-2}{2}}\choose{\frac{k-j}{2}}}\quad \forall\, k\in \N,\\
	e_k(r)&=\sum_{\substack{j=1\\\text{$k-j$ is even}}}^ka_j^{(2)}(r)\,\lambda^{\frac{k-j}{2}} {{\frac{k+j-2}{2}}\choose{\frac{k-j}{2}}} \quad \forall\, k\in \N,\\
	b_k(r)&= r^{n-1}\,\sum_{j=0}^k d_{k-j}(r)\,e_j(r) \quad k=0,1,2,\ldots
\end{aligned}\end{equation}
\begin{lemma}
	\label{lem_product_amp}
	Given any $\tau_1,\tau_2$ satisfying \eqref{tau_1}-\eqref{tau_2}, there holds
	$$ r^{n-1}\,A^{(1)}_{\tau_1}(r)\,A^{(2)}_{\tau_2}(r) = \sum_{k=0}^{N} \frac{b_k(r)}{\tau^k} + B_\tau(r), \quad \forall\, r\in [\varepsilon_0,2\varepsilon_0],$$ where $N=\lfloor\frac{\varepsilon_0\tau}{32e} \rfloor $ and the following statements are satisfied
	\begin{itemize}
		\item [(i)]{$b_0(r) =1$.}
		\item[(ii)]{$\|b_k\|_{L^{\infty}((\varepsilon_0,2\varepsilon_0))} \leq C_{\varepsilon_0}\,\left(\frac{4k}{\varepsilon_0}\right)^{k} $ for all $k\in \N$,}
		\item[(iii)]{$\|B_\tau\|_{L^{\infty}((\varepsilon_0,2\varepsilon_0))}\leq C_{\varepsilon_0}\,e^{-\frac{\tau \varepsilon_0}{64e}},$}
	\end{itemize}
for some $C_{\varepsilon_0}>0$ that is independent of $k$ and $\tau$.
\end{lemma}

\begin{proof}
	We begin by claiming that for each $r\in [\varepsilon_0,2\varepsilon_0]$ and any $\tau_1,\tau_2$ as in \eqref{tau_1}-\eqref{tau_2} there holds
	\begin{equation}
		\label{A_tau1_exp}
		A_{\tau_1}^{(1)}(r)= \sum_{k=0}^N \frac{d_k(r)}{\tau^k} + D_\tau(r), \quad \text{and}\quad 	A^{(2)}_{\tau_2}(r)= \sum_{k=0}^N \frac{e_k(r)}{\tau^k} + E_\tau(r),
	\end{equation}
where $d_0(r)=e_0(r)=r^{-\frac{n-1}{2}}$ and we have the following estimates:
\begin{equation}
	\label{A_tau1_principal_est}
	 \|d_k\|_{L^{\infty}((\varepsilon_0,2\varepsilon_0))} \leq C_{\varepsilon_0} \left(\frac{2k}{\varepsilon_0}\right)^k \quad \text{and} \quad  \|e_k\|_{L^{\infty}((\varepsilon_0,2\varepsilon_0))} \leq C_{\varepsilon_0} \left(\frac{2k}{\varepsilon_0}\right)^k
\end{equation}
for all $k=1,2,\ldots$ and
\begin{equation}
	\label{A_tau1_remainder_est}
	\|D_{\tau}\|_{L^{\infty}((\varepsilon_0,2\varepsilon_0))} \leq C_{\varepsilon_0}e^{-\frac{\tau \varepsilon_0}{32e}} \quad \text{and} \quad 	\|E_{\tau}\|_{L^{\infty}((\varepsilon_0,2\varepsilon_0))} \leq C_{\varepsilon_0}e^{-\frac{\tau \varepsilon_0}{32e}}  ,
\end{equation}
for some constant $C_{\varepsilon_0}>0$ that is independent of $\tau$. We will only prove the estimates \eqref{A_tau1_principal_est}--\eqref{A_tau1_remainder_est} for $A_{\tau_1}^{(1)}$ as the estimates for $A_{\tau_2}^{(2)}$ follow analogously. Using Taylor series remainder estimates for the map $x\mapsto(1+x)^{-k}$, it is straightforward to write for each $k=1,\ldots,N$, 
$$ (\tau+\frac{\lambda}{\tau})^{-k} =\tau^{-k}\,(1+\frac{\lambda}{\tau^2})^{-k}= \sum_{\ell=0}^{\lfloor\frac{N-k}{2}\rfloor}(-\lambda)^{\ell} {{k+\ell-1}\choose{\ell}} \tau^{-k-2\ell} + L_{k} \quad k=1,\ldots,N,$$
where
$$L_k= \tau^{-k}\,{{m+k}\choose{m+1}}\,(1+x)^{-k-m-1} \left(\frac{\lambda}{\tau^2}\right)^{m+1}\, \quad m=\lfloor\frac{N-k}{2}\rfloor$$ 
where $x\in [1,1+\frac{\lambda}{\tau^2}]$ is some value depending on $\tau$ and $\lambda$. As $k\leq N$, and $\lambda \in [0,1]$, it follows that
\begin{equation}\label{error_sum_est} |L_k| \leq \left(\frac{2}{\tau}\right)^N \quad \text{for all $k=1,\ldots,N$}.\end{equation} 
We may therefore rewrite the above calculations as follows
$$ \left(\tau+\frac{\lambda}{\tau}\right)^{-k}= \sum_{j=k}^N \frac{s_{k,j}}{\tau^j} + L_k,$$
where given any $k=1,\ldots,N$ and $j=k,\ldots,N$, we define
\begin{equation}\label{def_s_k_j} 
	s_{k,j}=\begin{cases}
(-\lambda)^{\frac{j-k}{2}} {{\frac{j+k-2}{2}}\choose{\frac{j-k}{2}}} \quad &\text{if $j-k$ is even}\\
0 \quad &\text{if $j-k$ is odd.}
\end{cases}\end{equation} 
As $|\lambda| \leq 1$ and $k\leq j$, we note that
\begin{equation}\label{s_L_est}
|s_{k,j}|\leq 2^j \quad \text{for $k=1,\ldots,N$ and $j=k,\ldots,N$}.
	\end{equation}
We have
\begin{multline*} A_{\tau_1}^{(1)}(r) = a_0^{(1)}(r)+\sum_{k=1}^N \frac{a_k^{(1)}(r)}{\tau_1^k} =a_0^{(1)}(r)+ \sum_{k=1}^N\sum_{j=k}^N a_k^{(1)}(r)\,s_{k,j}\,\tau^{-j} + \sum_{k=1}^N a^{(1)}_k(r)\,L_k\\
 = a_0^{(1)}(r)+\sum_{k=1}^N(\sum_{j=1}^k a^{(1)}_j(r)\,s_{j,k})\,\tau^{-k} + \sum_{k=0}^N a^{(1)}_k(r)\,L_k
\end{multline*}
Recalling Definition~\ref{def_amplitudes} together with Lemma~\ref{lem_amplitudes} and the estimate \eqref{error_sum_est} and using the fact that $2^{-k}k!\leq k^k$, for all $r\in[\varepsilon_0,2\varepsilon_0]$, we write
\begin{multline}\label{A_tau1_remainder} 
	|\sum_{k=1}^N a_k^{(1)}(r)\,L_k|\leq C\sum_{k=1}^N r^{-\frac{n-1}{2}-k} 2^{-k}(k+1)!\, \left(\frac{2}{\tau}\right)^{N} \leq C_{\varepsilon_0}\,(N+1) \, \,\left(\frac{2N}{\tau\varepsilon_0}\right)^N\\
\leq  C_{\varepsilon_0}\,\,(N+1) \, \,\left(\frac{1}{4e}\right)^{N} \leq C_{\varepsilon_0}\, e^{-\frac{\tau \varepsilon_0}{32e}},
\end{multline}
for some constant $C_{\varepsilon_0}$ that is independent of $\tau$ that may change from line to line. Next, recalling the definition of $s_{k,j}$ from \eqref{def_s_k_j} (we caution the reader that the role of $j$ and $k$ have been switched here), we note that for each $k=1,\ldots,N,$   
\begin{equation}
\label{eq_d_def}
\sum_{j=1}^k a_j^{(1)}(r)\,s_{j,k}=\sum_{\substack{j=1\\\text{$k-j$ is even}}}^ka_j^{(1)}(r)\,(-\lambda)^{\frac{k-j}{2}} {{\frac{k+j-2}{2}}\choose{\frac{k-j}{2}}} 
	\end{equation}
and therefore we obtain the claimed form \eqref{A_tau1_exp} for $A^{(1)}_{\tau_1}$ with the correct remainder estimate \eqref{A_tau1_remainder_est} that follows directly from \eqref{A_tau1_remainder}. Using \eqref{s_L_est}, the estimate \eqref{A_tau1_principal_est} for $d_k$ can also be obtained as follows:
\begin{multline*}
	\|d_k\|_{L^{\infty}((\varepsilon_0,2\varepsilon_0))} \leq  2^k\sum_{j=1}^k \|a_j^{(1)}\|_{L^{\infty}((\varepsilon_0,2\varepsilon_0))}
	\leq  2^{k}\sum_{j=1}^k \left(\frac{1}{\varepsilon_0}\right)^{\frac{n-1}{2}+j} \,2^{-j}(j+1)!\\
	 \leq 2^{k}\, \left(\frac{1}{\varepsilon_0}\right)^{\frac{n-1}{2}+k}\,(k+1)!\,\sum_{j=1}^{\infty}2^{-j}\\
	\leq  e\,\left(\frac{1}{\varepsilon_0}\right)^{\frac{n-1}{2}}\,2^{k}\,k^k\, \left(\frac{1}{\varepsilon_0}\right)^{k}\leq C_{\varepsilon_0} \left(\frac{2k}{\varepsilon_0}\right)^k.
\end{multline*}
We can derive the estimate \eqref{A_tau1_principal_est} for $e_k$ by applying similar arguments.
We are now ready to prove the claims of the lemma. The first claim (i) is trivial and follows from the definition \eqref{d_e_seq}. For (ii), observe from the estimates \eqref{A_tau1_principal_est} that there is a constant $C_{\varepsilon_0}$ independent of $\tau$ such that
$$ \|b_k\|_{L^{\infty}((\varepsilon_0,2\varepsilon_0))} \leq C_{\varepsilon_0}\,\sum_{j=0}^k \left(\frac{2(k-j)}{\varepsilon_0}\right)^{k-j}\,\left(\frac{2j}{\varepsilon_0}\right)^j\leq  C_{\varepsilon_0}\,(k+1)\,\left(\frac{2k}{\varepsilon_0}\right)^k \leq C_{\varepsilon_0}\,\left(\frac{4k}{\varepsilon_0}\right)^k.$$

Next, let us prove the claim (iii). Using \eqref{A_tau1_exp}-\eqref{A_tau1_remainder_est} together with Lemma~\ref{lem_amp_est}, we write
\begin{equation}\label{A_prod_est_1}
	r^{n-1}A^{(1)}_{\tau_1}(r)\,A^{(2)}_{\tau_2}(r) = r^{(n-1)} \sum_{k=0}^{2N}\sum_{j=\max\{0,k-N\}}^{\min\{k,N\}} \tau^{-k}\,d_{j}(r)\, e_{k-j}(r)+\tilde{B}_\tau, 
\end{equation}
where 
$$\|\tilde B_\tau\|_{L^{\infty}((\varepsilon_0,2\varepsilon_0))} \leq C_{\varepsilon_0}\,e^{-\frac{\tau \varepsilon_0}{32e}},$$
for some constant $C_{\varepsilon_0}$ that is independent of $\tau$. Next, note that 
\begin{multline}
r^{n-1}\sum_{k=0}^{2N}\sum_{j=\max\{0,k-N\}}^{\min\{k,N\}} \tau^{-k}\,d_{j}(r)\, e_{k-j}(r)\\
= \sum_{k=0}^{N} \tau^{-k}\,b_k(r)+\underbrace{r^{n-1}\sum_{k=N+1}^{2N}\sum_{j=k-N}^{N}\tau^{-k}d_j(r)e_{k-j}(r)}_{\hat{B}_\tau}.\end{multline}
Using the estimates \eqref{A_tau1_principal_est} we write
\begin{multline*}
	|\hat{B}_\tau| \leq C_{\varepsilon_0}\,\sum_{k=N+1}^{2N}\sum_{j=k-N}^{N}\tau^{-k}\, \left(\frac{2j}{\varepsilon_0}\right)^j\,\left(\frac{2(k-j)}{\varepsilon_0}\right)^{k-j}\\
	\leq C_{\varepsilon_0}\,\sum_{k=N+1}^{2N}\sum_{j=k-N}^{N}\tau^{-k}\,\left(\frac{2k}{\varepsilon_0}\right)^{k}\leq C_{\varepsilon_0}\,\sum_{k=N+1}^{2N}\tau^{-k}\,N\, \left(\frac{2k}{\varepsilon_0}\right)^{k}\\
	\leq C_{\varepsilon_0}\,\sum_{k=N+1}^{2N} N\, e^{-k}\leq C_{\varepsilon_0} \frac{e}{e-1}\, e^{-\frac{N+1}{2}} \leq C_{\varepsilon_0} \frac{e}{e-1}\, e^{-\frac{\varepsilon_0\tau}{64e}}.
\end{multline*}
We now define 
$$ B_\tau = \tilde B_\tau + \hat B_\tau,$$
and note that 
$$ r^{n-1}A_{\tau_1}(r)\,A_{\tau_2}(r) = \sum_{k=0}^{N} \tau^{-k}\underbrace{r^{n-1}\sum_{j=0}^{k}\,d_{j}(r)\, e_{k-j}(r)}_{b_k(r)}\,+B_\tau(r),$$
and that (iii) follows directly from the estimates for $ \tilde B_\tau$ and  $\hat B_\tau$.
\end{proof}

We return to \eqref{Q_est}, written as follows
$$
	\left|\int_{\varepsilon_0}^{2\varepsilon_0} Q(r) \,e^{-2\tau r}\,\left(\sum_{k=0}^N \frac{b_k(r)}{\tau^k}+ B_\tau(r)\right)\,dr\right| \leq C_{\varepsilon_0}\, e^{-(2\varepsilon_0+2\varepsilon_2)\tau}.
$$
Using the triangle inequality, recalling that $\varepsilon_2<\frac{\varepsilon_1}{32e}<\frac{\varepsilon_0}{128e}$ together with (iii) from Lemma~\ref{lem_product_amp} we write
\begin{equation}
\label{key_iden_asymp}
	\left|\int_{\varepsilon_0}^{2\varepsilon_0} Q(r) \,e^{-2\tau r}\,\left(\sum_{k=0}^N \frac{b_k(r)}{\tau^k}\right)\,dr\right|\leq C\, e^{-(2\varepsilon_0+2\varepsilon_2)\tau},
\end{equation}
for some constant $C>0$ that depends only on $\varepsilon_0$ and $\|q\|_{L^{\infty}(M)}.$ To complete the proof of Theorem~\ref{thm_density_loc} we need to study the above equation as $\tau \to \infty$. We start with a definition.
\begin{definition}
	\label{int_r}
	Given any $f\in L^{\infty}((\varepsilon_0,2\varepsilon_0))$, we define
	$$ \mathcal I f(r) = \int_{\varepsilon_0}^r f(s)\,ds,$$
	and write $$ \mathcal I^{k}f(r) = (\underbrace{\mathcal I \circ \mathcal I\circ \ldots\circ \mathcal I}_{\text{$k$ times}}f)(r),$$
	for all $k=1,2,\ldots$. We also define $\mathcal I^{0}f=f$.
\end{definition}
It is straightforward to write an explicit formula for $\mathcal I^{k}f$  as follows
\begin{equation}
	\label{int_comp}
	\mathcal I^kf(r) = \int_{\varepsilon_0}^r \frac{(r-s)^{k-1}}{(k-1)!}f(s)\,ds, \quad \forall\, k\in \N.
\end{equation}
Consequently, we also deduce the following bound
\begin{equation}
	\label{int_comp_bound}
	\|\mathcal I^kf\|_{L^{\infty}((\varepsilon_0,2\varepsilon_0))} \leq \frac{\varepsilon_0^k}{k!}\,\|f\|_{L^{\infty}((\varepsilon_0,2\varepsilon_0))} \quad k=0,1,\ldots.
\end{equation}
Using Definition~\ref{int_r}, we have the following lemma.
\begin{lemma}
	\label{lem_int_comp}
	There exists a constant $C>0$ depending only on $\varepsilon_0$ and $\|Q\|_{L^{\infty}((\varepsilon_0,2\varepsilon_0))}$ such that
	$$
	\sum_{k=0}^N\,\left|\int_{\varepsilon_0}^{2\varepsilon_0} Q(r)\,e^{-2\tau r}\,\tau^{-k}\,b_k(r)\,dr -\int_{\varepsilon_0}^{2\varepsilon_0} 2^k\,e^{-2\tau r}\,\mathcal I^k(Qb_k)(r)\,dr \right| \leq C\, e^{-(2\varepsilon_0+2\varepsilon_2)\tau},
$$
for all $\tau$ as in \eqref{tau_range}.

\end{lemma}
\begin{proof}
	Note that when $k=0$, the left hand side of the above summation is zero since $b_0=1$. Therefore, it suffices to consider the summation from $k=1$ to $k=N$. We begin by considering 
	$$ \int_{\varepsilon_0}^{2\varepsilon_0} Q(r)\,e^{-2\tau r}\,\tau^{-k}\,b_k(r)\,dr \quad \forall\, k\in \N$$
	and apply integration by parts repeatedly until all the negative powers, $\tau^{-k}$, are removed. This yields
	\begin{multline}\label{Q_int_comp_formula}
		 \int_{\varepsilon_0}^{2\varepsilon_0} Q(r)\,e^{-2\tau r}\,\tau^{-k}\,b_k(r)\,dr -\int_{\varepsilon_0}^{2\varepsilon_0}
		 2^k\,e^{-2\tau r}\,(\mathcal I^kQb_k)(r)\,dr\\
		 = e^{-4\varepsilon_0\tau} \left(\sum_{j=1}^k\frac{2^{k-j}}{\tau^j} (\mathcal I^{k-j+1}Qb_k)|_{r=2\varepsilon_0} \right)
		\end{multline}
	Recalling the estimate (ii) in Lemma~\ref{lem_product_amp}, together with \eqref{int_comp_bound} and the fact that 
	$$k^k \leq e^k\,k! \quad \forall\, k\in \N,$$ 
	we write
	\begin{multline*}
	 \left|\sum_{j=1}^k\frac{2^{k-j}}{\tau^j} (\mathcal I^{k-j+1}Qb_k)(2\varepsilon_0) \right|\leq C_{\varepsilon_0}\|Q\|_{L^{\infty}((\varepsilon_0,2\varepsilon_0))}\sum_{j=1}^k 2^{k-j}\tau^{-j}\,\frac{\varepsilon_0^{k-j+1}}{(k-j+1)!}\left(\frac{4k}{\varepsilon_0}\right)^k\\
	 \leq C_{\varepsilon_0}\|Q\|_{L^{\infty}((\varepsilon_0,2\varepsilon_0))}\sum_{j=1}^k {(2\varepsilon_0)}^{k-j}\tau^{-j}\,\frac{1}{(k-j+1)!}\left(\frac{4e}{\varepsilon_0}\right)^k\, k!\\
	 =C_{\varepsilon_0}\|Q\|_{L^{\infty}((\varepsilon_0,2\varepsilon_0))}\sum_{j=1}^k {(2\varepsilon_0)}^{k-j}\tau^{-j}\,(j-1)!\,{{k}\choose{j-1}}\left(\frac{4e}{\varepsilon_0}\right)^k\\
	 \leq C_{\varepsilon_0}\|Q\|_{L^{\infty}((\varepsilon_0,2\varepsilon_0))}\sum_{j=1}^k {(2\varepsilon_0)}^{k-j}\tau^{-j}\,j^j\,2^{k}\,\left(\frac{4e}{\varepsilon_0}\right)^k\\
	  \leq  C_{\varepsilon_0}\|Q\|_{L^{\infty}((\varepsilon_0,2\varepsilon_0))}(16e)^k\sum_{j=1}^k \left(\frac{j}{2\tau\varepsilon_0}\right)^j.
	\end{multline*}
Noting that $k\leq N$, we have the bounds
$$ (16e)^k \leq (16e)^N \leq e^{\frac{\varepsilon_0\tau}{8e}},$$
and 
$$ 
\sum_{j=1}^k \left(\frac{j}{2\tau\varepsilon_0}\right)^j \leq \sum_{j=1}^k \left(\frac{N}{2\tau\varepsilon_0}\right)^j \leq \sum_{j=1}^k \left(\frac{1}{64e}\right)^j <1.
$$
Therefore, 
$$
 \left|\sum_{j=1}^k\frac{2^{k-j}}{\tau^j} (\mathcal I^{k-j+1}Qb_k)(2\varepsilon_0) \right|\leq C_{\varepsilon_0}\|Q\|_{L^{\infty}((\varepsilon_0,2\varepsilon_0))}\,e^{\frac{\varepsilon_0\tau}{8e}}.
$$
Combining the above inequality with \eqref{Q_int_comp_formula} we deduce that
\begin{multline*}
	\sum_{k=0}^N\left|\int_{\varepsilon_0}^{2\varepsilon_0} Q(r)\,e^{-2\tau r}\,\tau^{-k}\,b_k(r)\,dr -\int_{\varepsilon_0}^{2\varepsilon_0} 2^k\,e^{-2\tau r}\,\mathcal I^k(Qb_k)(r)\,dr \right| \\
	\leq \sum_{k=1}^NC_{\varepsilon_0}\|Q\|_{L^{\infty}((\varepsilon_0,2\varepsilon_0))}\,e^{-4\varepsilon_0\tau}\, e^{\frac{\varepsilon_0\tau}{8e}} \leq C\, e^{-(2\varepsilon_0+2\varepsilon_2)\tau},
\end{multline*}
for some constant $C>0$ that depends only on $\varepsilon_0$ and $\|Q\|_{L^{\infty}((\varepsilon_0,2\varepsilon_0))}$.
\end{proof}
Applying Lemma~\ref{lem_int_comp}, the bound \eqref{key_iden_asymp} can now be rewritten as follows:
\begin{equation}
	\label{key_iden_asymp_1}
	\left|\int_{\varepsilon_0}^{2\varepsilon_0} e^{-2\tau r}\left(\sum_{k=0}^N2^k\,\mathcal I^k(Qb_k)(r)\right)\,dr \right| \leq C\, e^{-(2\varepsilon_0+2\varepsilon_2)\tau},
\end{equation}
for some constant $C>0$ that depends only on $\varepsilon_0$ and $\|Q\|_{L^{\infty}((\varepsilon_0,2\varepsilon_0))}$. Next, we have the following lemma.
\begin{lemma}
	\label{lem_key_asymp}
	There is a constant $C>0$ depending only on $\|Q\|_{L^{\infty}((\varepsilon_0,2\varepsilon_0))}$ and $\varepsilon_0$ such that 
	\begin{equation}\label{key_asymp_est} \left|\int_{\varepsilon_0}^{\varepsilon_0+\varepsilon_2} e^{-2\tau r}\left(\sum_{k=0}^N2^k\,\mathcal I^k(Qb_k)(r)\right)\,dr \right| \leq C\, e^{-(2\varepsilon_0+2\varepsilon_2)\tau}, \end{equation}
	for all $\tau$ as in \eqref{tau_range}.
\end{lemma}
\begin{proof}
We claim that
	\begin{equation}
		\label{key_asymp_est_1}
		\left|\int_{\varepsilon_0+\varepsilon_2}^{\varepsilon_0+\frac{\varepsilon_0}{9e}} e^{-2\tau r}\left(\sum_{k=0}^N2^k\,\mathcal I^k(Qb_k)(r)\right)\,dr \right| \leq C\, e^{-(2\varepsilon_0+2\varepsilon_2)\tau}
	\end{equation}
for some constant $C>0$ that depends only on $\|Q\|_{L^{\infty}((\varepsilon_0,2\varepsilon_0))}$ and $\varepsilon_0$. 
Indeed, note that on the interval $I=[\varepsilon_0,\varepsilon_0+\frac{\varepsilon_0}{9e}]$ there holds:
$$ \|\sum_{k=0}^N2^k\,\mathcal I^k(Qb_k)\|_{L^{\infty}(I)} \leq \|Q\|_{L^{\infty}(I)}\,C_{\varepsilon_0}\,\sum_{k=0}^N 2^k \left(\frac{4k}{\varepsilon_0}\right)^k\,\left(\frac{\varepsilon_0}{9e}\right)^k\frac{1}{k!},$$
where we have used the bound (ii) in Lemma~\ref{lem_product_amp} together with the analogue of \eqref{int_comp_bound} but restricted on the interval I. Using $k^k \leq e^k k!$, the above estimate reduces to 
$$ \|\sum_{k=0}^N2^k\,\mathcal I^k(Qb_k)\|_{L^{\infty}(I)} \leq \|Q\|_{L^{\infty}(I)}\,C_{\varepsilon_0}\,\sum_{k=0}^N  \left(\frac{8e}{\varepsilon_0}\right)^k\,\left(\frac{\varepsilon_0}{9e}\right)^k\leq 9\|Q\|_{L^{\infty}(I)}\,C_{\varepsilon_0}
$$
The claim \eqref{key_asymp_est_1} follows immediately from the above inequality. Next, we claim that 
\begin{equation}
	\label{key_asymp_est_2}
	\left|\int_{\varepsilon_0+\frac{\varepsilon_0}{9e}}^{2\varepsilon_0} e^{-2\tau r}\left(\sum_{k=0}^N2^k\,\mathcal I^k(Qb_k)(r)\right)\,dr \right| \leq C\, e^{-(2\varepsilon_0+2\varepsilon_2)\tau}
\end{equation}
for some constant $C>0$ that depends only on $\|Q\|_{L^{\infty}((\varepsilon_0,2\varepsilon_0))}$ and $\varepsilon_0$.  In order to prove \eqref{key_asymp_est_2} we first note that on the interval $I'=[\varepsilon_0+\frac{\varepsilon_0}{9e},2\varepsilon_0]$ there holds
\begin{multline*} \|\sum_{k=0}^N2^k\,\mathcal I^k(Qb_k)\|_{L^{\infty}(I')} \leq \|Q\|_{L^{\infty}(I')}\,C_{\varepsilon_0}\,\sum_{k=0}^N 2^k \left(\frac{4k}{\varepsilon_0}\right)^k\,(\varepsilon_0)^k\frac{1}{k!}\\
	\leq \|Q\|_{L^{\infty}(I')}\,C_{\varepsilon_0}\,\sum_{k=0}^N \left(\frac{8e}{\varepsilon_0}\right)^k\,(\varepsilon_0)^k \leq  \|Q\|_{L^{\infty}(I')}\,C_{\varepsilon_0}\,\frac{8e}{8e-1}\,(8e)^{N}\\
	\|Q\|_{L^{\infty}(I')}\,C_{\varepsilon_0}\,\frac{8e}{8e-1}\,e^{4N} \leq \|Q\|_{L^{\infty}(I')}\,C_{\varepsilon_0}\,\frac{8e}{8e-1}\,e^{\frac{\varepsilon_0}{8e}\tau}. 
\end{multline*}
Therefore, 
$$ 	\left|\int_{\varepsilon_0+\frac{\varepsilon_0}{9e}}^{2\varepsilon_0} e^{-2\tau r}\left(\sum_{k=0}^N2^k\,\mathcal I^k(Qb_k)(r)\right)\,dr \right|\leq \|Q\|_{L^{\infty}(I')}\,C_{\varepsilon_0}\,\frac{8e}{8e-1}\,e^{\frac{\varepsilon_0}{8e}\tau}\,e^{-2\varepsilon_0\tau-\frac{2\varepsilon_0}{9e}\tau}$$ 
The claim \eqref{key_asymp_est_2} now follows immediately from the fact that 
$$ \frac{\varepsilon_0}{8e}-\frac{2\varepsilon_0}{9e}\leq -2\varepsilon_2.$$
The proof of the estimate \eqref{key_asymp_est} now follows from combining \eqref{key_asymp_est_1}--\eqref{key_asymp_est_2} with \eqref{key_iden_asymp_1}.
\end{proof}

\begin{lemma}
	\label{lem_B_def}
	Let 
	\begin{equation}
		\label{def_D}
		D=\{(r,s) \in \R^2\,:\, r\in (\varepsilon_0,\varepsilon_0+\varepsilon_2)\quad s\in (\varepsilon_0,r) \}.
	\end{equation}
	There exists a bounded function $B \in L^{\infty}(D)$ defined by
	\begin{equation}\label{def_B}
		 B(r,s)=\lim_{m\to \infty}\sum_{k=1}^m \frac{2^k}{(k-1)!}\,(r-s)^{k-1}\,b_k(s),\end{equation}
	where the convergence is with respect to the norm $\|\cdot\|_{L^{\infty}(D)}$. Moreover, defining 
	$$ E_m(r,s) = B(r,s)-\sum_{k=1}^m \frac{2^k}{(k-1)!}\,(r-s)^{k-1}\,b_k(s),$$
	there holds 
		\begin{equation}\label{E_est} 
			\|E_m\|_{L^{\infty}(D)} \leq C_{\varepsilon_0}\, e^{-m} \quad \forall\, m\in \N,\end{equation}
		for some constant $C_{\varepsilon_0}>0$ depending only on $\varepsilon_0$.
\end{lemma}
\begin{proof}
	To justify the formal definition \eqref{def_B}, it suffices to note that given any $(r,s)\in D$ and any $m\in \N$, there holds
	$$ \sum_{k=1}^m \left|\frac{2^k}{(k-1)!}\,(r-s)^{k-1}\,b_k(s)\right| \leq C_{\varepsilon_0}\sum_{k=1}^m 2^k \,\frac{\varepsilon_2^{k-1}}{(k-1)!} \left(\frac{4k}{\varepsilon_0}\right)^k
		\leq \frac{C_{\varepsilon_0}}{\varepsilon_2}\sum_{k=1}^\infty k\,\left(\frac{8e\varepsilon_2}{\varepsilon_0}\right)^k <\infty,
	$$
	where we have used the fact that $\frac{8e\varepsilon_2}{\varepsilon_0}<1$. As the right hand side of the above bound is independent of $m\in \N$, we conclude that the definition \eqref{def_B} makes sense. In order to prove the remainder estimate \eqref{E_est} for the partial sums, we note that given any $m\in \N$ and any $\ell \geq m+1$, there holds 
	$$ 
	\sum_{k=m+1}^\ell \left|\frac{2^k}{(k-1)!}\,(r-s)^{k-1}\,b_k(s)\right| \leq \frac{C_{\varepsilon_0}}{\varepsilon_2}\sum_{k=m+1}^\ell \left(\frac{16e\varepsilon_2}{\varepsilon_0}\right)^k\leq C\, e^{-m},
	$$
	for some constant $C>0$ independent of $m$ and $\ell$, where we are using the fact that $\frac{16e\varepsilon_2}{\varepsilon_0}<e^{-1}$.
\end{proof}

\begin{lemma}
	\label{lem_final_est}
	There exists a constant $C>0$ depending only on $\|Q\|_{L^{\infty}((\varepsilon_0,2\varepsilon_0))}$ and $\varepsilon_0$ such that 
	$$ \left|\int_{\varepsilon_0}^{\varepsilon_0+\varepsilon_2} e^{-2\tau r}\,\left( Q(r) + \int_{\varepsilon_0}^r B(r,s)\,Q(s)\,ds \right)\,dr\right| \leq C\, e^{-(2\varepsilon_0+2\varepsilon_2)\tau},$$
	for all $\tau>0$ as in \eqref{tau_range}.
\end{lemma}

\begin{proof}
In view of \eqref{int_comp}, there holds
$$ \sum_{k=1}^N 2^k\,\mathcal I^k(Qb_k)(r)= \int_{\varepsilon_0}^rQ(s)\left(\sum_{k=1}^N \frac{2^k}{(k-1)!}(r-s)^{k-1}\,b_k(s)\right)\,ds.$$
Therefore given every $r\in (\varepsilon_0,\varepsilon_0+\varepsilon_1)$ we have
\begin{multline}\label{Q_diff_B} \left|\int_{\varepsilon_0}^{r}Q(s)\,B(r,s)\,ds-\sum_{k=1}^N 2^k\,\mathcal I^k(Qb_k)(r)\right|=\left|\int_{\varepsilon_0}^rQ(s)\,E_N(r,s)\,ds\right|\\
	\leq C\|E_N\|_{L^{\infty}(D)}
	\leq C e^{-N}\leq C e^{-2\tau\varepsilon_2},\end{multline}
where we are using \eqref{E_est} with $m=N$ in the last step together with the fact that 
$$N=\lfloor \frac{\tau\varepsilon_0}{32e}\rfloor\geq 2\varepsilon_2\tau.$$
The proof is completed by combining the above bound \eqref{Q_diff_B} with \eqref{key_asymp_est}.
\end{proof}

\begin{lemma}
	\label{Q_zero}
	The function $Q \in L^{\infty}((\varepsilon_0,2\varepsilon_0))$ vanishes identically on the interval $(\varepsilon_0,\varepsilon_0+\varepsilon_2)$.
\end{lemma}

\begin{proof}
	Let us define $H\in L^{\infty}((0,\varepsilon_2))$ via
	$$ H(r)= Q(\varepsilon_0+\varepsilon_2-r)+\int_{\varepsilon_0}^{\varepsilon_0+\varepsilon_2-r} B(\varepsilon_0+\varepsilon_2-r,s)\,Q(s)\,ds.$$
	In view of Lemma~\ref{lem_final_est}, there holds:
	\begin{equation}
		\label{analytic_est_0}
		\left|\int_{0}^{\varepsilon_2} e^{2\tau r}\, H(r)\,dr\right| \leq C,
	\end{equation}
	for all $\tau> \min\{n,\frac{64e}{\varepsilon_0}\}$. Here, the constant $C>0$ depends only on $\varepsilon_0$ and $\|Q\|_{L^{\infty}((\varepsilon_0,2\varepsilon_0))}$. It is straightforward to see that the above estimate implies that  
		\begin{equation}
		\label{analytic_est}
		\left|\int_{0}^{\varepsilon_2} e^{2\tau r}\, H(r)\,dr\right| \leq C_1
	\end{equation}
	for all $\tau \in \R$, where the constant $C_1>0$ is independent of $\tau \in \R$. Let us now define  
	$$ F(z) = \int_{0}^{\varepsilon_2} e^{2zr}\, H(r)\,dr \quad \forall\, z\in \C.$$
	The function $F(z)$ is an entire holomorphic function. Given any $z\in \C$, there holds
	$$ |F(z)| \leq \|H\|_{L^{\infty}((0,\varepsilon_2))}\, e^{2\varepsilon_2|z|} \quad \forall\, z\in \C.$$
	The function $F(z)$ is uniformly bounded on the imaginary axis $\textrm{Re}(z)=0$ and finally uniformly bounded on the real axis $\text{Im}(z)=0$, thanks to \eqref{analytic_est}. Applying the Phragm\'{e}n-Lindel\"{o}f principle in each quadrant of the complex plane, we conclude that $F(z)$ must be uniformly bounded and hence identically equal to a constant. Finally, as 
	$$ \lim_{x\to -\infty}F(x)=0,$$
	we conclude that 
	$$F(z) =0 \quad \forall\, z\in \C.$$
	Next, plugging $z=-\frac{1}{2}\textrm{i}\xi$ with $\xi \in \R$, we deduce that
	$$ \int_0^{\varepsilon_2} e^{-\textrm{i}\xi r}H(r)\,dr=0 \quad \forall\, \xi \in\R.$$
	The above identity implies that $H(r)=0$ for almost every $r\in (0,\varepsilon_2)$. This in turn implies that
	$$ Q(r) + \int_{\varepsilon_0}^r B(r,s)\,Q(s)\,ds =0, \quad \text{for a.e. $r\in (\varepsilon_0,\varepsilon_0+\varepsilon_2)$.}$$
	By boundedness of $B$ on the set $D$ (see \eqref{def_D}), we obtain that
	$$ |Q(r)|\leq \|B\|_{L^{\infty}(D)}\,\int_{\varepsilon_0}^{r}|Q(s)|\,ds \quad \text{for a.e. $r\in (\varepsilon_0,\varepsilon_0+\varepsilon_2)$}.$$
	Finally, applying Gr\"{o}nwall's inequality we conclude that $Q$ must vanish identically on the interval $(\varepsilon_0,\varepsilon_0+\varepsilon_2)$.
\end{proof}
The proof of Theorem~\ref{thm_density_loc} is completed thanks to the following lemma.
\begin{lemma}
	\label{lem_proof_of_thm}
	The function $q=q_1-q_2\in L^{\infty}(M)$ vanishes identically on the neighbourhood $\mathcal V$ of the point $p$ defined via polar coordinates by
	\begin{equation}\label{vanish_set} (0,T)\times (\varepsilon_0,\varepsilon_0+\varepsilon_2)\times \mathbb S_+^{n-1}.\end{equation}
\end{lemma}

\begin{proof}
	In view of Lemma~\ref{Q_zero} together with the definition of Q (see \eqref{def_Q}), we note that
		$$ 
		0=\int_\delta^{T-\delta}\int_{\mathbb S_+^{n-1}} q(t,r,\theta)\,e^{4\lambda t} Y_{\sigma_1,\omega_1}(\theta)\,Y_{\sigma_2,\omega_2}(\theta)\,dt\,dV_{\mathbb S_+^{n-1}} \quad \text{for a.e. $r\in (\varepsilon_0,\varepsilon_0+\varepsilon_2)$}.
		$$
	As the expression depends analytically on the parameter $\lambda$ and since it vanishes for all $\lambda \in (0,1)$ we conclude that the expression must vanish for all $\lambda \in \C$. Noting that $q=0$ on $\mathcal D_\delta$, we obtain that 
	\begin{equation}\label{eq_final_thm3} 0=\int_{\mathbb S_+^{n-1}} q(t,r,\theta)\, Y_{\sigma_1,\omega_1}(\theta)\,Y_{\sigma_2,\omega_2}(\theta)\,dV_{\mathbb S_+^{n-1}} \quad \text{ a.e. $(t,r)\in (0,T)\times  (\varepsilon_0,\varepsilon_0+\varepsilon_2)$}.\end{equation}
		We will first consider the case when $n=2$. In this case, we recall from \eqref{spherical_harmonics} that the role of $\omega$ is redundant. We choose
		$$ \sigma_1=\sigma_2=\sigma \in [0,1],$$ 
		 so that equation \eqref{eq_final_thm3} reduces to
		$$
		0=\int_{0}^\pi q(t,r,\theta)\, e^{2\sigma\theta}\,d\theta \quad \text{for a.e. $(t,r)\in (0,T)\times  (\varepsilon_0,\varepsilon_0+\varepsilon_2)$}.
		$$
		Since the latter expression depends analytically on $\sigma$ and since it vanishes for all $\sigma\in[0,1]$ we conclude that it must also vanish for all $\sigma \in \C$.  Therefore the function $q$ must vanish identically on the set \eqref{vanish_set}. 
		
		Next, we consider the remaining case when the dimension satisfies $n\geq 3$. In this case, we recall that equation \eqref{eq_final_thm3} is satisfied for all $Y_{\sigma_j,\omega_j}$, $j=1,2$, satisfying \eqref{spherical_harmonics} for some $\sigma=\sigma_j \in [0,1]$ and $\omega=\omega_j\in C^{\infty}(\p \mathbb S_+^{n-1})$ that additionally satisfies \eqref{def_omega}. Noting that the solutions to \eqref{spherical_harmonics} are invariant under scaling, we deduce that \eqref{def_omega} is redundant and therefore the equation is satisfied for any $\omega_j\in C^{\infty}(\p \mathbb S_+^{n-1})$, $j=1,2$. In view of this observation, we choose $\sigma_2=0$ and $\omega_2$ to be identical to one on $\p \mathbb S_+^{n-1}$ so that $Y_{\sigma_2,\omega_2}=1$ on $\mathbb S_+^{n-1}$. This implies that
		$$
		0=\int_{\mathbb S_+^{n-1}} q(t,r,\theta)\, Y(\theta)\,dV_{\mathbb S_+^{n-1}} \quad \text{for a.e. $(t,r)\in (0,T)\times (\varepsilon_0,\varepsilon_0+\varepsilon_2)$},
		$$ 
		 for all smooth functions $Y \in C^{\infty}(\overline{\mathbb S^{n-1}_+})$ that satisfy
		 $$ -\Delta_{\mathbb S^{n-1}_+} Y+\sigma^2Y=0 \quad \text{on $\mathbb S^{n-1}_+$}\quad \text{for some $\sigma \in[0,1]$}.$$
		 Applying Proposition~\ref{prop_appendix} with $(M,g)$ being the hemisphere $\mathbb S^{n-1}_+$, we conclude that $q$ must vanish identically on the set \eqref{vanish_set}. 
	\end{proof}

\section{Proofs of Theorem~\ref{thm_lin}--Theorem~\ref{thm_density}}
\label{sec_global}
The goal of this section is to derive the global claim in Theorem~\ref{thm_density}  from its local counterpart, namely Theorem~\ref{thm_density_loc}. It is then straightforward to see that Theorem~\ref{thm_lin} follows immediately from combining Lemma~\ref{lem_density} and Theorem~\ref{thm_density}. We will prove Theorem~\ref{thm_density} via a layer stripping argument where we combine a Runge approximation property for solutions to heat equations combined with repeatedly applying our local result as we march inwards into the domain $\Omega$. This method of deriving global completeness results for products of solutions from local ones has already been used in the context of elliptic equations, see e.g. \cite[Section 2]{DKSU}. We will closely follow their approach here, with some minor modifications. We will denote the complement of $\Gamma$ by $\widetilde \Gamma$, that is to say,
$$ \widetilde \Gamma= \p \Omega \setminus \Gamma.$$
Our starting point is a Runge approximation property, as follows. We remark that such approximation properties are rather standard, see e.g. \cite{RS} and the references therein.
\begin{lemma}[Runge's approximation property]
	\label{lem_Runge}
	Let $I=(a,b)\subset \R$ be an open non-empty interval. Let $\Omega_1 \subset \Omega_2\subset \R^n$ be connected bounded domains with smooth boundaries such that $\Omega_2 \setminus \overline \Omega_1$ is connected. Suppose that $S:= \p \Omega_1 \cap \p \Omega_2$ is a non-empty closed set (see Figure \eqref{fig2}). Let us define the four subspaces $\mathcal W_{j,\pm} \subset L^2(I\times \Omega_1)$, $j=1,2,$ via 
	$$ \mathcal W_{1,\pm} = \{ u \in C^{\infty}(\overline{I\times \Omega_1}):\, \pm\p_t u-\Delta u=0 \quad \text{on $I\times \Omega_1$ and $u|_{I\times S}=0$}\}, 
	$$
	and
	$$
	\mathcal W_{2,\pm} = \{ u|_{I\times \Omega_1}:\, u  \in C^{\infty}(\overline{I\times\Omega_2)}\quad \pm\p_t u-\Delta u=0 \quad \text{on $I\times \Omega_2$ and $u|_{I\times S}=0$}\} 
	$$
	The subspace $\mathcal W_{2,\pm}$  is dense in $\mathcal W_{1,\pm}$ with respect to the $L^2(I\times \Omega_1 )$ topology. 
\end{lemma}
\begin{proof}
	We only prove that the subspace $\mathcal W_{2,+}$  is dense in $\mathcal W_{1,+}$ with respect to the $L^2(I\times \Omega_1 )$ topology, as the claim for $\mathcal W_{1,+}$ follows analogously. 	It suffices to prove that given any $v\in L^2(I\times \Omega_1)$, if $v$ satisfies
	\begin{equation}
		\label{Runge_1}
		\int_{I\times \Omega_1} v\, u_2 \,dt\,dx =0 \quad \forall\, u_2\in \mathcal W_{2,+},
	\end{equation}
then it must also satisfy
	\begin{equation}
	\label{Runge_2}
	\int_{I\times \Omega_1} v\, u_1 \,dt\,dx =0 \quad \forall\, u_1\in \mathcal W_{1,+}.
\end{equation}
	\begin{figure}[!ht]
  \centering
  \includegraphics[width=0.6\textwidth]{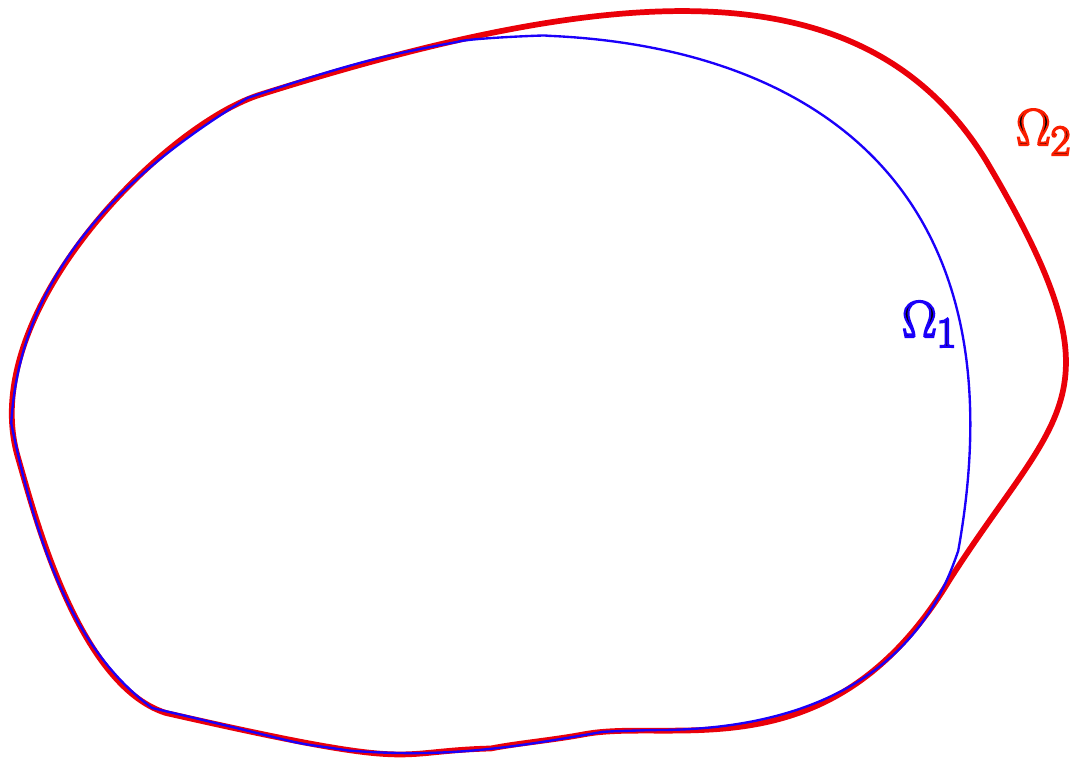}
  \caption{Figure for $\Omega_1$, $\Omega_2$  }
  \label{fig2}
\end{figure}

To show this, we suppose that $v\in L^2(I\times \Omega_1)$ satisfies \eqref{Runge_1}. We define $\tilde v \in L^2(I \times \Omega_2)$ via
$$
\tilde v= \begin{cases}
	v 
		&\text{on  $I\times \Omega_1$},
		\\
	0 & \text{on $I \times (\Omega_2\setminus \overline\Omega_1)$}.
	\end{cases}
$$
and let $w\in H^{1,2}(I \times \Omega_2)$ to be the unique solution to 
\begin{equation}\label{Runge_3}
	\begin{aligned}
		\begin{cases}
			-\p_t w - \Delta w=\tilde v 
			&\text{on  $(a,b)\times \Omega_2$},
			\\
			w=0 & \text{on $(a,b)\times \p \Omega_2$},
			\\
			w|_{t=b}=0 & \text{on $\Omega$}.
		\end{cases}
	\end{aligned}
\end{equation}
Then, equation \eqref{Runge_1} reduces as follows
\begin{multline*}
	\int_{I\times \Omega_1} v\, u_2 \,dt\,dx =\int_{I\times \Omega_2}  \tilde v\, u_2\,dt\,dx = \int_{I\times \Omega_2} (-\p_t-\Delta)w\, u_2\,dt\,dx\\
	= \int_{\Omega_2} w(a,x)\,u_2(a,x)\,dx -\int_{I\times \tilde{S}} u_2\,\p_\nu w \,dt\,dx,
\end{multline*}
where $\tilde S= \p \Omega_2 \setminus \p \Omega_1$. As the latter equation holds for any $u_2\in \mathcal W_{2,+}$, we deduce that 
\begin{equation}\label{Runge_4} w(a,x)=0 \quad \forall\, x\in \Omega_2 \quad \text{and} \quad\p_\nu w|_{I\times \tilde S}=0.\end{equation}
As $w|_{I\times \tilde S}$ is also zero, and since 
$$ -\p_t w - \Delta w=0 \quad \text{on $I\times (\Omega_2\setminus \overline{\Omega_1})$},$$
it follows from the unique continuation principle for heat equations that $w$ must vanish identically on the set $I\times (\Omega_2\setminus \overline{\Omega_1})$. In particular, there holds
\begin{equation}
	\label{Runge_5}
	w|_{I\times (\p \Omega_1\setminus \p\Omega_2)}= \p_\nu w|_{I\times (\p \Omega_1\setminus \p\Omega_2)}=0.
\end{equation}
Therefore, given any $u_1 \in \mathcal W_{1,+}$ there holds
$$
	\int_{I\times \Omega_1} v\, u_1 \,dt\,dx= \int_{I\times \Omega_1} (-\p_t-\Delta)w\, u_1\,dt\,dx=0,$$
	where we have used integration by parts in the last step together with \eqref{Runge_3}--\eqref{Runge_5}.
\end{proof}
We are ready to provide the layer stripping approach in proving how the local Theorem~\ref{thm_density_loc} yields the global Theorem~\ref{thm_density}. Our proof here will closely follow that of \cite[Theorem 1.1]{DKSU}.
\begin{proof}[Proof of Theorem~\ref{thm_density}]
	We begin by considering an arbitrary point $x_1 \in \Omega$. Let $\gamma:[0,1]\to \overline{\Omega}$ be a smooth non-self-intersecting curve joining the point $p\in \Gamma$ to the point $x_1\in \Omega$ such that $\gamma(0)=p$, $\dot{\gamma}(0)$ is the interior unit normal to $\p \Omega$ at the point $p$, $\gamma(1)=q$ and that $\gamma(s)\in \Omega$ for all $s\in (0,1]$. Next, we consider $\epsilon>0$ to be sufficiently small (to be quantified later) and subsequently, for each $s\in [0,1]$, we consider the closed tubular neighbourhoods
	$$ \Theta_{\epsilon}(s) =  \{ x\in \overline\Omega\,:\, \textrm{dist}(x,\gamma([0,s]))\leq \epsilon\},$$
	of the curve ending at $\gamma(s)$, and consider a closed subset on $[0,1]$ defined by
	$$ J= \{ s\in [0,1]\,:\, \text{$q(t,x)$ vanishes a.e. on $x\in \Theta_\epsilon(s)$ and $t\in [0,T]$}\}.$$ To conclude the proof of the theorem, it suffices to show that the closed set $J\subset [0,1]$ is also open and not empty, and hence $J=[0,1]$. 
	
	To this end, we start by noting that if $\epsilon>0$ is sufficiently small, in view of the fact that \eqref{density_eq} is satisfied, the set $J$ is nonempty, thanks to Theorem~\ref{thm_density_loc}. If $s\in J$ and $\epsilon$ is sufficiently small, then we
	may suppose $\p \Theta_\epsilon(s) \cap \p \Omega \subset \Gamma$ and that $\Omega\setminus \Theta_\epsilon(s)$ can be smoothed out into an open subset $\Omega_1 \subset \Omega$ with smooth boundary such that
	$$ \Omega \setminus \Theta_\epsilon(s) \subset \Omega_1 \quad \text{and} \quad \widetilde \Gamma \subset\p\Omega\cap \p\Omega_1,$$
	where we recall that $\widetilde \Gamma= \p\Omega \setminus \Gamma$. We also augment the set $\Omega$ by smoothing out the set $\Omega \cup B(p,\epsilon)$ into an open set $\Omega_2$ with smooth boundary in such a way that
	$$ \widetilde \Gamma \subset \p \Omega_1\cap \p \Omega\subset \p\Omega_2\cap \p \Omega.$$
	As $s\in J$, we know that $q(t,x)$ vanishes for a.e. $x\in \Theta_\epsilon(s)\cap \Omega$ and $t\in [0,T]$. Therefore, in view of Lemma~\ref{lem_Runge} together with \eqref{density_eq} we deduce that 
		\begin{equation}\label{density_eq_alt} \int_\delta^{T-\delta}\int_{\Omega_1} q\, u_1\,u_2 \,dx\,dt =0,\end{equation}
	for any pair $u_1\in C^{\infty}([\delta,T]\times \overline\Omega_1)$ and $u_2 \in C^{\infty}([0,T-\delta]\times \overline\Omega_1)$ that respectively satisfy 
	$$ \p_t u_1 -\Delta u_1=0 \quad \text{on $(\delta,T)\times \Omega_1$},$$ 
	with 
	$$\supp (u_1|_{(\delta,T)\times \p \Omega_1})\subset (\delta,T)\times (\p\Omega_1\setminus \p\Omega_2),$$
	and 
	$$ \p_t u_2 +\Delta u_2=0 \quad \text{on $(0,T-\delta)\times \Omega_1$},$$ 
	with 
	$$\supp (u_2|_{(0,T-\delta)\times \p \Omega_1})\subset (0,T-\delta)\times (\p\Omega_1\setminus \p\Omega_2).$$
	Next, by applying the local completeness result for products of solutions to heat equations stated in Theorem~\ref{thm_density_loc} (with $\Omega$ and $\Gamma$ replaced by $\Omega_1$ and $\p\Omega_1\setminus \p\Omega_2$ in its statement and everything else kept the same), we deduce that $q$ vanishes on a neighbourhood of $\p \Omega_1 \setminus (\p\Omega_1\cap \p\Omega_2)$. Hence, $q$ vanishes on a slightly larger neighbourhood $\Theta_\epsilon(\tau)$, $\tau>s$ proving that $J$ is indeed an open set. Finally, since the choice of $x_1\in \Omega$ is arbitrary, we conclude that $q$ vanishes identically on $(0,T)\times \Omega$. 
	\end{proof}

We conclude this section by repeating that the proof of Theorem~\ref{thm_lin} can now be trivially obtained thanks to Lemma~\ref{lem_density} and Theorem~\ref{thm_density}.

\section{Proof of Theorem~\ref{thm_semi}}
\label{sec_semi}
We assume that the condition \eqref{ts1} is fulfilled and we will prove that this condition implies that $a_1=a_2$. For this purpose, in view of condition (H), it would be enough to show by iteration that the condition
\begin{equation}\label{ts1c} \partial_\mu^ka_1(t,x,0)=\partial_\mu^ka_2(t,x,0),\quad (t,x)\in M\end{equation}
holds true for all $k\geq2$. Fix $m\geq2$, $\varepsilon_1,\ldots,\varepsilon_m\in[0,1)$ and $f_1,\ldots,f_m\in\ C^\infty_0((0,T)\times\Gamma)$. Fix also $\varepsilon:=(\varepsilon_1,\ldots,\varepsilon_m)\in\mathbb R^m$.
For $j=1,2$, let us consider the initial boundary value problem
\begin{equation}\label{heat_semi1}
	\begin{aligned}
		\begin{cases}
			\p_t u_{j,\varepsilon} - \Delta u_{j,\varepsilon} + a_j(t,x,u_{j,\varepsilon})=0 
			&\text{on  $M=(0,T)\times \Omega$},
			\\
			u_{j,\varepsilon}=\sum_{k=1}^m \varepsilon_jf_j & \text{on $\Sigma$},
			\\
			u_{j,\varepsilon}|_{t=0}=0 & \text{on $\Omega$}.
		\end{cases}
	\end{aligned}
\end{equation}
According to \cite[Proposition 3.1.]{KLL}, there exists $\kappa_j>0$ depending on $f_1,\ldots,f_m$, $T$ $a_j$ and $\Omega$ such that for $|\varepsilon|<\kappa_j$ the problem \eqref{heat_semi1} admits a unique solution $u_{j,\varepsilon}\in C^{1+\frac{\alpha}{2},2+\alpha}(\overline{M})$. Moreover, the map $\varepsilon\mapsto u_{j,\varepsilon}\in C^{1+\frac{\alpha}{2},2+\alpha}(\overline{M})$ will be smooth for $|\varepsilon|<\kappa_j$. We use an induction on $m$ based on analyzing the function
$$\frac{1}{m!}\partial_{\varepsilon_1}\ldots\partial_{\varepsilon_m}u_{j,\varepsilon}|_{\varepsilon=0},$$
to deduce the claim \eqref{ts1c} for each $k=m=2,3,\ldots$. We begin with the base case $m=2$. Let us consider for $\varepsilon=(\varepsilon_1,\varepsilon_2)$ as above, the function
$$
v_j = \frac{1}{2}\partial_{\varepsilon_1}\partial_{\varepsilon_2} u_{j,\varepsilon}|_{\varepsilon=0}.
$$
Note that $v_j$ satisfies the linear inhomogeneous equation
\begin{equation}
	\label{second_der_eq}
		\begin{aligned}
			\begin{cases}
				\p_t v_j - \Delta v_j= -\p^2_\mu a_j(t,x,0)\,u_1 u_2 
				&\text{on  $M=(0,T)\times \Omega$},
				\\
				v_{j}=0 & \text{on $\Sigma$},
				\\
				v_{j}|_{t=0}=0 & \text{on $\Omega$},
			\end{cases}
		\end{aligned}
\end{equation}
with $u_j$, $j=1,2$, solving the free linear heat equation \eqref{eq_free} with Dirichlet boundary data $f=f_j$.  Combining this with the arguments used in \cite[Section 3.2 and Section 4 Step 1 of the proof of Theorem 2.1]{KLL} we deduce from the equality of the Dirichlet-to-Neumann maps, namely,
\begin{equation}
	 \label{ts1}
	 \mathcal N_{a_1}^\Gamma(f)=\mathcal N_{a_2}^\Gamma(f)
\end{equation}
applied to the function $f=\varepsilon_1 f_1+\varepsilon_2f_2$ that there holds
$$ \p_\nu v_1 |_{(0,T)\times \Gamma}= \p_\nu v_2|_{(0,T)\times \Gamma}.$$
Comparing the above identity with \eqref{eq_free_source}--\eqref{Frechet_der} we deduce that
\begin{equation}\label{ts1a}
\mathcal S_{q_1}^\Gamma(f_2)= \mathcal S_{q_2}^\Gamma(f_2) \quad \forall\, f_2\in C^{\infty}_0(((0,T)\times \Gamma)),
\end{equation}
where 
$$ q_1(t,x) = \p^2_\mu a_1(t,x,0)u_1(t,x) \quad \text{and} \quad q_2(t,x) = \p^2_\mu a_2(t,x,0)u_1(t,x).$$
Now, we fix $\delta\in (0,\frac{T}{2})$ and $s\in (\delta,T-\delta)$ and consider a function $\zeta \in C^{\infty}_0((0,T))$ so that $\zeta(t)>0$ for $t\in (\delta,s)$ and that $\zeta(t)=0$ for $t\leq \delta$ and $t\geq s$. We also fix a non-zero function $g \in C^{\infty}(\p \Omega)$ so that $g(x)=0$ for all $x\in \p\Omega\setminus \Gamma$ and $g(x)\geq 0$ for all $x\in \Gamma$. Let $w$ be the unique solution to \eqref{eq_free} with $f(t,x)=\zeta(t)\,g(x)$ for all $(t,x)\in \Sigma$. By null controllability for the heat equation, see e.g. \cite[Theorem 2]{LR}, there exists $h\in C^{\infty}_0((s,T-\delta)\times \Gamma)$ such that the unique solution to
\begin{equation}\label{eq_free_s}
	\begin{aligned}
		\begin{cases}
			\p_t \tilde{w} - \Delta \tilde{w}=0 
			&\text{on  $(s,T)\times \Omega$},
			\\
			\tilde{w}=h & \text{on $(s,T)\times \p \Omega$},
			\\
			\tilde{w}(s,x)=w(s,x) & \text{for $x\in\Omega$},
		\end{cases}
	\end{aligned}
\end{equation}
satisfies $\tilde{w}(T-\delta,x)=0$, $x\in\Omega.$
Combining this with the fact that $\tilde{w}|_{(T-\delta,T)\times\partial\Omega}\equiv0$, we obtain
\begin{equation}\label{tilde_w_zero} \tilde{w}(t,x) =0, \quad \forall\, (t,x)\in (T-\delta,T)\times \Omega.\end{equation}
We will now fix the function $f_1\in C^{\infty}_0((0,T)\times \Gamma)$ above to be defined by 
\begin{equation}\label{ff}f_1(t,x)=\zeta(t)g(x)+h(t,x),\quad (t,x)\in\Sigma.\end{equation}
Noticing that $u_1|_{(s,T)\times\Omega}$ solves \eqref{eq_free_s}, we deduce that $u_1|_{(s,T)\times\Omega}=\tilde{w}$ and \eqref{tilde_w_zero} implies that $u_1|_{(T-\delta,T)\times\Omega}\equiv0$. Moreover, recalling that $f_1|_{(0,\delta)\times\partial\Omega}\equiv0$, we deduce that $u_1(t,x)=0$ for $t\leq \delta$ and $x\in \Omega$. We therefore deduce that 
\begin{equation}\label{q_initial_vanish} q_1 = q_2=0 \quad \text{on $\mathcal D_\delta$}.\end{equation}
In view of \eqref{ts1a}--\eqref{q_initial_vanish}, we can apply Theorem~\ref{thm_lin} to conclude that $q_1=q_2$ for all $(t,x)\in M$. In addition, applying the strong maximum principle for parabolic equations (see e.g. \cite[Theorem 11 pp 375]{Ev}), we obtain
$$u_1(t,x)>0,\quad (t,x)\in(\delta,s)\times \Omega$$
and it follows that 
$$\partial_\mu^2a_1(t,x,0)=\partial_\mu^2a_2(t,x,0) \quad \forall\, (t,x) \in (\delta,s)\times \Omega.$$
Finally, we conclude that the claim \eqref{ts1c} holds for $k=2$ since $\delta\in (0,\frac{T}{2})$ and $s\in (\delta,T-\delta)$ are allowed to change arbitrarily.

We will now fix $m\geq2$ and assume that \eqref{ts1c} is fulfilled for all $k=2,\ldots,m-1$. We let $\varepsilon$ and $u_{j,\varepsilon}$, $j=1,2,$ be as above and note that the function
$$ v_j = \frac{1}{m!}\partial_{\varepsilon_1}\ldots\partial_{\varepsilon_m}u_{j,\varepsilon}|_{\varepsilon=0}$$
satisfies
\begin{equation}
		\label{m_der_eq}
	\begin{aligned}
		\begin{cases}
			\p_t v_j - \Delta v_j= -\p^m_\mu a_j(t,x,0)\,u_1 u_2\ldots u_m  + H_j(t,x),
			&\text{on  $M=(0,T)\times \Omega$},
			\\
			v_{j}=0 & \text{on $\Sigma$},
			\\
			v_{j}|_{t=0}=0 & \text{on $\Omega$},
		\end{cases}
	\end{aligned}
	\end{equation}
with $u_j$, $j=1,\ldots,m$, solving the problem \eqref{eq_free} with $f=f_j\in C^\infty_0((0,T)\times\Gamma)$ 
and where $H_1=H_2$, thanks to our induction hypothesis for $k\leq m-1$. It is straightforward to see that the condition \eqref{ts1} applied to $f=\sum_{k=1}^m\varepsilon_kf_k$ implies that there holds
$$ \p_\nu v_1 |_{(0,T)\times \Gamma}= \p_\nu v_2|_{(0,T)\times \Gamma}.$$
Comparing the above identity with \eqref{eq_free_source}--\eqref{Frechet_der} together with the fact that $H_{1}=H_2$, we deduce that
\begin{equation}\label{ts2a}
	\mathcal S_{q_1}^\Gamma(f_m)= \mathcal S_{q_2}^\Gamma(f_m) \quad \forall\, f_m \in C^{\infty}_0((0,T)\times \Gamma),
\end{equation}
where 
$$ q_j(t,x) =\p^m_\mu a_j(t,x,0)\,u_1(t,x)\ldots u_{m-1}(t,x) \quad j=1,2.$$
Following as in the case of $m=2$, we may choose 
$$ f_1=\ldots=f_{m-1}=\zeta(t)\,g(x)+h(t,x),$$
and apply Theorem~\ref{thm_lin} to conclude that \eqref{ts1c} is fulfilled for $k=m$. Therefore, \eqref{ts1c} holds true for all $k\geq2$ which completes the proof of the theorem.

\section{Completeness of fixed frequency waves}
Our goal in this section is to prove the completeness result that was used in the last step of our proof of Theorem~\ref{thm_density}. Precisely, we prove the following proposition.
\begin{proposition}
	\label{prop_appendix}
	Let $I=(a,b)\subset \R$ be a nonempty open interval and let $(M,g)$ be a smooth connected compact Riemannian manifold of dimension $n\geq 2$ with smooth boundary. Let $q\in L^{\infty}(M)$ satisfy 
	$$ \int_M q(x)\,u(x)\,dV_g=0,$$
	for all $u\in C^{\infty}(M)$ that satisfy 
	$$(-\Delta_g +\lambda)u=0 \quad  \text{on $M$},$$
	for some $\lambda \in I$. Then, $q\equiv 0$ on $M$.
\end{proposition}

Before presenting the proof of the above proposition, let us fix some notation. We recall that in local coordinates $(x^1,\ldots,x^n)$, the Laplace--Beltrami operator $\Delta_g$ is given by the expression
$$ \Delta_g u = \sum_{j,k=1}^n \frac{1}{\sqrt{\left|\det g\right|}}\frac{\p}{\p x^j}\left( \sqrt{\left|\det g\right|} \,g^{jk}\frac{\p u}{\p x^k} \right)\quad \forall\, u\in C^{\infty}(M),$$
where $\det g$ stands for the determinant of the metric tensor $g=(g_{jk})_{j,k=1}^n$ and $(g^{jk})_{j,k=1}^n$ stands for the components of the inverse matrix. We consider also the Laplace operator with homogeneous Dirichlet boundary condition on $(M,g)$ associated with the unbounded operator $A$ acting on $L^2(M)$ with domain $D(A)=H^1_0(M)\cap H^2(M)$ satisfying $Au=-\Delta_g u$, for all $u\in D(A)$. We will denote by $\{\lambda_k\}_{k=1}^{\infty}\subset (0,\infty)$ the  eigenvalues of $A$ written in strictly increasing order so that 
$$ 0<\lambda_1<\lambda_2<\ldots,$$
and write $d_k$ for the multiplicity of eigenvalue $\lambda_k$. For each $k\in \N$, we let $\{\phi_{k,j}\}_{j=1}^{d_k}\subset C^{\infty}(M)$be  an $L^2(M)$-orthonormal set for the eigenspace of $A$ associated to the eigenvalue $\lambda_k$ so that
$$ A \phi_{k,j} = \lambda_k\,\phi_{k,j}  \quad \forall\, k\in \N \quad j=1,\ldots,d_k.$$
 Next, we define for each $k \in \N$, the operator 
$$ S_k: C^{\infty}(\p M) \to L^2(M),$$
by the expression  
$$ (S_k f)(x) = \sum_{j=1}^{d_k} \left(\int_{\p M} f\,\p_\nu \phi_{k,j}\,dV_g\right)\, \phi_{k,j}(x),$$
where $\nu$ is the exterior unit normal vector field on $\p M$. 
\begin{proof}[Proof of Proposition~\ref{prop_appendix}]
Let $\mathbb A=\C \setminus \{\lambda_1,\lambda_2,\ldots\}$. Following \cite[Lemma 3.2.]{BCDKS}, for each $z \in \mathbb A$, and any $f\in C^{\infty}(\p M)$,we recall that   the boundary value problem

\begin{equation}\label{laplace}
	\begin{aligned}
		\begin{cases}
			(- \Delta_g-z) u_z^f=0 
			&\text{on  $M$},
			\\
			u_z^f=f & \text{on $\p M$},
		\end{cases}
	\end{aligned}
\end{equation}
admits a unique solution $u^f_z\in H^2(M)$ taking the form
$$ u_z^f(x) = \sum_{k=1}^{\infty}\frac{1}{\lambda_k-z}\,(S_kf)(x).$$
In addition, fixing $\mu\in \mathbb A$, one can check that $$u_z^f=u_\mu^f+(z-\mu)(A-z)^{-1}u_\mu^f,\quad z\in \mathbb A$$
and deduce from this representation that  the mapping 
$$ z \mapsto u_z^f \quad z \in \mathbb A,$$
is holomorphic on the set $\mathbb A$ in the sense of a map taking values in $L^2(M)$ with  poles at $\{\lambda_k:\ k\in\mathbb N\}$. Therefore, given any fixed $f\in C^{\infty}(\p M)$ the function 
$$ z\mapsto \int_{M} q \, u_z^f\,dV_g \quad z\in \mathbb A,$$
is  holomorphic. By the hypothesis of the proposition, the above expression vanishes for each $z\in I$. Therefore, by analytic continuation it must vanish for all $z\in \mathbb A$. Hence, 
$$ \int_M q\, u_z^f \,dV_g=0 \quad \forall\, f\in C^{\infty}(\p M) \quad \forall\, z\in \mathbb A.$$ 
Let $k\in \N$ be fixed. Multiplying the above expression with $\lambda_k-z$ and letting $z \to \lambda_k$, we conclude that there holds
$$ \int_M q \,S_k(f) \,dV_g =0 \quad \forall\, f\in C^{\infty}(\p M) \quad \forall\, k\in \N.$$
Therefore, given any $k\in \N$, we have
\begin{equation}\label{eq_appendix} \sum_{j=1}^{d_k} \left(\int_M q\,\phi_{k,j}\,dV_g\right)\, \left(\int_{\p M} f\p_\nu \phi_{k,j}\,dV_g\right)=0\quad  \forall\, f\in C^{\infty}(\p M)\quad \forall\, k\in \N.\end{equation} 
Moreover, in light of \cite[Lemma 2.1.]{CK}, we recall  that, for each $k\in \N$, the maps
\begin{equation}\label{set_indep}\p_\nu \phi_{k,1}|_{\p M},\p_\nu \phi_{k,2}|_{\p M},\ldots,\p_\nu \phi_{k,d_k}|_{\p M}\in L^2(\p M),\end{equation}
are linearly independent.  Combining this with the fact that in formula \eqref{eq_appendix} the map $f\in C^{\infty}(\p M)$ is arbitrary chosen, we deduce that
$$ \int_M q\,\phi_{k,j}\,dV_g=0 \quad \forall\, k\in \N \quad j=1,\ldots,d_k.$$
Finally, recalling that $\{\phi_{k,1},\ldots,\phi_{k,d_k}:\ k\in\mathbb N\}$ is an orthonormal basis of $L^2(M)$, we deduce that $q=0$ on $M$.
\end{proof}


\begin{thebibliography}{99}
	
	\bibitem{Bel1} 
	Belishev, M., {\em An approach to multidimensional inverse problems for the wave equation}, Dokl. AN SSSR \textbf{297} (1987) 524--527 (in Russian).
	
	\bibitem{Bel2}
	Belishev, M., {\em Recent progress in the boundary control method, Inverse Problems}, 23 (2007), R1--R67.
	
	\bibitem{BK92}
	Belishev, M., Kurylev Y., {\em To the reconstruction of a Riemannian manifold via its spectral data (BC-method)}, Comm. Partial Differential Equations, \textbf{17} (1992), 767--804. 	

	\bibitem{BCDKS}
	Bellassoued, M., Choulli,  M., Dos Santos Ferreira, D., Kian, Y., Stefanov, P.,
	A Borg-Levinson theorem for magnetic Schr\"{o}dinger operators on a Riemannian manifold, Ann. Inst. Fourier 71, No. 6, 2471--2517 (2021). 
		\bibitem{BU} 
		A. L. Bukhgeim and G. Uhlmann,  Recovering a potential from partial  Cauchy data, Commun. Partial Diff. Equat. \textbf{27} (3-4) (2002),  653-668.
	\bibitem{Calderon}
	Calder\'{o}n, A. P., On an inverse boundary value problem, Seminar on Numerical Analysis and its Applications to Continuum Physics (Rio de Janeiro, 1980), Soc. Brasil. Mat., Rio de Janeiro, 1980, 65–73, (Reprinted in Computational \& Applied Mathematics 25 (2016), no. 2--3, 133--138.)
	
	\bibitem{CK}
	Canuto, B., Kavian, O., Determining coefficients in a class of heat equations via boundary measurements. SIAM J. Math. Anal. 32(5), 963--986 (2001)
	
	\bibitem{CK1}
	Caro, P., Kian, Y., Determination of convection terms and quasi-linearities appearing
	in diffusion equations, arXiv preprint (2018).
	
	\bibitem{CK2}
	Choulli, M., Kian, Y., Logarithmic stability in determining the time-dependent zero order coefficient in a parabolic equation from a partial Dirichlet-to-Neumann map, Application to the determination of a nonlinear term. J. Math. Pures Appl. 9(114), 235--261 (2018).
	
	\bibitem{COY}
	Choulli, M., Ouhabaz, E. M., Yamamoto, M., Stable determination of a semilinear term in a parabolic equation, Commun. Pure Appl. Anal. 5 (3) (2006), 447--462.
	
	\bibitem{DKSU}
	Dos Santos Ferreira, D.,  Kenig, C., Sj\"{o}strand, J.,  Uhlmann, G., 
	On the linearized local Calder\'{o}n problem, Math. Res. Lett., 16 (6) (2009), pp. 955--970.
	
	
	%\bibitem{DKSU2}
	%Dos Santos Ferreira, D., Kenig, C., Salo, M., Uhlmann, G., Limiting Carleman weights and anisotropic inverse problems, Invent. Math. 178 (2009), no. 1, 119--171.
	
	%\bibitem{DKLS}
	%Dos Santos Ferreira, D., Kurylev, Y., Lassas, M., Salo, M., The Calder\'{o}n problem in transversally anisotropic geometries, J. Eur. Math. Soc. 18 (2016), no. 11, 2579--2626.
	

	\bibitem{DKLLS}
	Dos Santos Ferreira D., Kurylev Y., Lassas, M., Liimatainen, T., Salo, M.,
	The linearized Calder\'{o}n problem in transversally anisotropic geometries
	Int. Math. Res. Not. (2022). 
	
	\bibitem{Ev}
	 Evans, L. C.,  Partial differential equations, American Mathematical Society, volume 19, 2010.
	
	\bibitem{FIS}
	Fisher, R.A., The wave of advance of advantageous genes. Annals of Eugenics, 7:355--369, 1937.
		
	\bibitem{Fei}
	Feizmohammadi, A., An inverse boundary value problem for isotropic nonautonomous heat flows, Math. Ann. 388, 1569--1607 (2024).
	
	\bibitem{FKU}
	Feizmohammadi, A., Kian, Y., Uhlmann, G., An inverse problem for a quasilinear convection–diffusion equation,
	Nonlinear Analysis, Volume 222, 2022, 112921.
	
	\bibitem{FO}
	Feizmohammadi, A., Oksanen, L., An inverse problem for a semi-linear elliptic equation in Riemannian geometries, J. Differ. Equ., 269 (6) (2020), pp. 4683--4719.
	
	\bibitem{GH}
	Garde, H., Hyv\"{o}nen, N., Linearised Calder\'{o}n problem, Reconstruction and Lipschitz stability for infinite-dimensional spaces of unbounded perturbations, SIAM J. Math. Anal., in press (2024).
	
	\bibitem{Gelfand}
	Gel'fand, I., Some aspects of functional analysis and algebra, In Proceedings of the International Congress of Mathematicians (Amsterdam, 1954), Vol. 1, pp. 253--276. Erven P. Noordhoff N. V., Groningen; North-Holland Publishing Co., Amsterdam, 1957.
	
	\bibitem{GST}
	Guillarmou, C., Salo, M, Tzou, L., The linearized Calder\'{o}n problem on complex manifolds, Acta Math. Sin. Engl. Ser., 35 (6) (2019), pp. 1043--1056.
	
	%\bibitem{HT}
	%Hassell, A., Tao, T., Upper and lower bounds for normal derivatives of Dirichlet eigenfunctions, Math. Res. Lett. 9 (2002), 289--305.

	%\bibitem{IUY}
	%Imanuvilov, O., Uhlmann, G., Yamamoto, M., The Calder\'{o}n problem with partial data in two dimensions, J. Amer. Math. Soc. 23 (2010), no. 3, 655--691.

	\bibitem{Is0}
	Isakov, V., Inverse Problems for Partial Differential Equations, Volume 127, Springer-Verlag, Berlin, Heidelberg, 2006.
	
	\bibitem{Is1}
 	Isakov, V., On uniqueness in inverse problems for semilinear parabolic equations. Arch. Rat. Mech. Anal. 124, 1--12 (1993).

	\bibitem{Is2} 
	Isakov, V., Uniqueness and stability in multidimensional inverse problems. Inverse Prob. 9, 579--621 (1993).
	
	\bibitem{Is3}
	Isakov, V., Uniqueness of recovery of some quasilinear Partial differential equations, Commun. Partial Diff. Eqns., 26 (2001), 1947--1973.
	
	\bibitem{Is4}
	Isakov, V., Completness of products of solutions and some inverse problems for PDE, J. Diff. Equat., 92 (1991), 305--316.
	
	\bibitem{KS}
	Kenig, C., Salo, M., The Calder\'{o}n problem with partial data on manifolds and applications, Anal. PDE 6 (2013), no. 8, 2003--2048.
	
	\bibitem{KSU}
	 Kenig, C., Sj\"{o}strand, J., Uhlmann, G., The Calder\'{o}n problem with partial data, Ann. of Math. (2), 165 (2007), no. 2, 567--591.
	
	\bibitem{KKU}
	Kian, Y., Krupchyk, K., Uhlmann, G., Partial data inverse problems for quasilinear conductivity equations, Math. Ann. 385 (2023), no. 3--4, 1611--1638.
	
	\bibitem{KLL}
	 Kian, Y., Liimatainen, T., Lin, Y.H., On determining and breaking the gauge class in inverse problems for reaction-diffusion equations, Forum Math. Sigma, Vol. 12, e25 (2024), 1--42.
	
	\bibitem{KiUh}
	Kian, Y., Uhlmann, G. Recovery of Nonlinear Terms for Reaction Diffusion Equations from Boundary Measurements. Arch Rational Mech Anal 247, 6 (2023).
	
	\bibitem{KLS}
	Krupchyk, K., Liimatainen, T., Salo, M., Linearized Calder\'{o}n problem and exponentially accurate quasimodes for analytic manifolds, Adv. Math. 403 (2022), Paper No. 108362, 43 pp.
	
	\bibitem{KU1}
	Krupchyk, K., Uhlmann, G., Partial data inverse problems for semilinear elliptic equations with gradient nonlinearities, Math. Res. Lett. 27 (2020), no. 6, 1801--1824.
	
	\bibitem{KU2}
	Krupchyk, K., Uhlmann, G., A remark on partial data inverse problems for semilinear elliptic equations, Proc. Amer. Math. Soc. 148 (2020), no. 2, 681--685.
	
	\bibitem{KuSa}
	Kumar Sahoo, S., Salo, M., The linearized Calder\'{o}n problem for polyharmonic operators, Journal of Differential Equations, 360
	(2023), 407--451.

	\bibitem{KLU}
	Kurylev, Y., Lassas, M., Uhlmann, G., Inverse problems for Lorentzian manifolds and non-linear hyperbolic equations, Invent. Math. 212 (2018), no. 3, 781--857.

	\bibitem{LLLS}
	Lassas, M., Liimatainen, T., Lin, Y.-H., Salo, M., Inverse problems for elliptic equations with power type nonlinearities, J. Math. Pures Appl. (9), 145 (2021), pp. 44--82.
	
	\bibitem{LLLS2}
	Lassas, M., Liimatainen, T., Lin, Y.-H., Salo, M., Partial data inverse problems and simultaneous recovery of boundary and coefficients for semilinear elliptic equations, Rev. Mat. Iberoam. 37 (2021), no. 4, 1553–1580.

	\bibitem{LR}
	 Lebeau, G. and  Robbiano, G., Controle exact de l'\'{e}quation de la chaleur, Communications in Partial Differential Equations 20 (1995), no. 1-2, 335--356.

	\bibitem{LM}
	Lions, J.-L., Magenes, E., Non homogeneous boundary value problems and applications, vol. II. Springer Verlag, Berlin (1972).

	\bibitem{NW}
	Newell, A. C., Whitehead, J. A., Finite bandwidth, finite amplitude convection. Journal Fluid Mechanics, 38:118--139, 1987.

	\bibitem{RS}
	R\"{u}land A., Salo, M., Quantitative Runge Approximation and Inverse Problems, International Mathematics Research Notices, Volume 2019, Issue 20, (2019), pp 6216--6234.

	\bibitem{SU}
	Sj\"{o}strand, J., Uhlmann, G., Local analytic regularity in the linearized Calder\'{o}n problem, Anal. PDE, 9 (3) (2016), pp. 515--544.
	
	\bibitem{SyUh}
	Sylvester, J., Uhlmann, G., A global uniqueness theorem for an inverse boundary value problem, Ann. of Math. (2) 125 (1987), no. 1, 153--169.
	
	\bibitem{Uhl}
	Uhlmann, G., Electrical impedance tomography and Calder´on’s problem, Inverse Problems, 25 (2009), no. 12, 123011, 39 pp.
	
	\bibitem{Vol14}
	Volpert, V., Elliptic Partial Differential Equations: Volume 2: Reaction-Diffusion Equations, volume 104. Springer, 2014.
	
	\bibitem{YBZFK38}
	Zeldovich, Y. B., Frank-Kamenetsky, D. A., A theory of thermal propagation of flame.
	Acta Physicochim URSS, 9:341--350, 1938.
	
	\end{thebibliography}
\end{document}